\documentclass[reqno, british]{amsart}

\usepackage{geometry}
\usepackage{mathrsfs}
\usepackage{enumitem}
\usepackage{amsthm}
\usepackage{amsmath}
\usepackage{amssymb}
\usepackage{hyperref}
\usepackage{graphicx, xypic}
\usepackage{float}

\makeatletter

\theoremstyle{plain}
\newtheorem{thm}{\protect\theoremname}[section]

\theoremstyle{plain}
\newtheorem{prop}[thm]{\protect\propositionname}
\theoremstyle{definition}
\newtheorem{defn}[thm]{\protect\definitionname}
\theoremstyle{plain}
\newtheorem{lem}[thm]{\protect\lemmaname}
\theoremstyle{definition}
\newtheorem{example}[thm]{\protect\examplename}
\theoremstyle{remark}
\newtheorem{rem}[thm]{\protect\remarkname}
\theoremstyle{plain}
\newtheorem{cor}[thm]{\protect\corollaryname}
\theoremstyle{plain}

 \newlist{casenv}{enumerate}{4}
 \setlist[casenv]{leftmargin=*,align=left,widest={iiii}}
 \setlist[casenv,1]{label={{\itshape\ \casename} \arabic*.},ref=\arabic*}
 \setlist[casenv,2]{label={{\itshape\ \casename} \roman*.},ref=\roman*}
 \setlist[casenv,3]{label={{\itshape\ \casename\ \alph*.}},ref=\alph*}
 \setlist[casenv,4]{label={{\itshape\ \casename} \arabic*.},ref=\arabic*}

\makeatother

\usepackage{babel}
\providecommand{\definitionname}{Definition}
\providecommand{\examplename}{Example}
\providecommand{\lemmaname}{Lemma}
\providecommand{\remarkname}{Remark}
\providecommand{\theoremname}{Theorem}
\providecommand{\casename}{Case}
\providecommand{\corollaryname}{Corollary}
\providecommand{\factname}{Fact}
\providecommand{\propositionname}{Proposition}

\newcommand\Cref[1]{{Corollary~\ref{#1}}}
\newcommand\Lref[1]{{Lemma~\ref{#1}}}
\newcommand\Tref[1]{{Theorem~\ref{#1}}}

\newcommand\Eref[1]{{Example~\ref{#1}}}

\global\long\def\gr{\mathrm{gr}}
\global\long\def\R{\mathcal{R}}

\global\long\def\layer#1#2{\overset{\left[#2\right]}{}#1}
\global\long\def\zero{\layer{0}{0}}

\global\long\def\one{\layer{0}{1}}
\global\long\def\minus{\layer{0}{-1}}
\global\long\def\zeroset#1{\overset{\left[0\right]}{}#1}
\global\long\def\ELTrop{\mathrm{ELTrop}}
\global\long\def\Trop{\mathrm{Trop}}

\global\long\def\ELT#1#2{\mathscr{R}\left(#2,#1\right)}
\global\long\def\minf{-\infty}

\global\long\def\L{\mathscr{L}}
\global\long\def\F{\mathscr{F}}

\global\long\def\t{\tau}

\global\long\def\quo#1#2{\raisebox{.2em}{\ensuremath{#1}}\left/\raisebox{-.2em}{\ensuremath{#2}}\right.}

\global\long\def\symmdash{\succeq_\circ}

\newcommand{\RR}{\mathbb{R}}      
\newcommand{\CC}{\mathbb{C}}

\DeclareMathOperator{\sign}{sign}
\DeclareMathOperator{\ELTsubmatrixrank}{ELT-submatrix-rank}
\DeclareMathOperator{\ELTkapranovrank}{ELT-Kapranov-rank}
\DeclareMathOperator{\ELTrank}{ELT-rank}
\DeclareMathOperator{\rank}{rank}
\DeclareMathOperator{\Kapranovrank}{Kapranov-rank}

\begin{document}
\title{Exploded Layered Tropical Linear Algebra}
\date{\today}

\author[Guy Blachar]{Guy Blachar}
\address{Department of Mathematics, Bar-Ilan University, Ramat-Gan 52900,
Israel.} \email{\href{mailto:blachag@biu.ac.il}{blachag@biu.ac.il}}

\author[Erez Sheiner]{Erez Sheiner}
\address{Department of Mathematics, Bar-Ilan University, Ramat-Gan 52900,
Israel.} \email{\href{mailto:erez@math.biu.ac.il}{erez@math.biu.ac.il}}

\thanks{This article contains work from Erez Sheiner's Ph.D.\ Thesis, which was accepted on 1.1.16, and from Guy Blachar's M.Sc.\ Thesis, both submitted to the Math Department at Bar-Ilan University. Both works were carried under the supervision of Prof.\ Louis Rowen from Bar-Ilan University, to whom we thank deeply for his help and guidance.}
\thanks{The authors would also like to thank the referee for his helpful comments.}



\begin{abstract}
Exploded layered tropical (ELT) algebra is an extension of tropical algebra with a structure of layers. These layers allow us to use classical algebraic results in order to easily prove analogous tropical results. Specifically we study the connection between the ELT determinant and linear dependency, and use a generalized version of Kapranov Theorem proved in \cite{DSS} (called the Fundamental Theorem).\\
In this paper we prove that an ELT matrix is singular if and only if its rows are linearly dependent and that the row rank and submatrix rank of an ELT matrix are equal. We also define an ELT rank for a tropical matrix, and prove that it is equal to its Kapranov rank. In addition, we formalize the concept of ELT inner products, and prove ELT versions of some known theorems such as Cauchy-Schwarz inequality.
\end{abstract}

\maketitle

\tableofcontents
\addtocontents{toc}{~\hfill\textbf{Page}\par}

\section{Introduction}

Tropical linear algebra, also known as Max-Plus linear algebra, has been studied for more than 50 years (ref.~\cite{B}). While tropical geometry mainly deals with geometric combinatorial problems, tropical linear algebra deals with algebraic non-linear combinatorial problems (for instance, the assignment problem \cite{K}). Tropical linear algebra may also be used as a mean to study the tropical algebraic geometry (for instance, the tropical resultant). Notable work in this field can be found at \cite{B, DSS, IR4, IR3, S}.\\

The main results in tropical geometry tend to use a combinatorial approach, and a great volume of work has been done in order to create equivalent algebraic definitions. Basic notions such as \linebreak matrix rank (\cite{DSS, Beasley2015}), bases (\cite{izhakian2013}), varieties (\cite{izhakian2011}), polynomial factorization (\cite{sheiner2012}) and congruences (\cite{perri2013}) are studied in different settings, and different theorems are discovered in each. In this paper we study the linear algebra induced by viewing tropical mathematics as a valuation of Puiseux series. Although motivated by Puiseux series, our results work in a more general setting.\\

In tropical linear algebra, there are many definitions for linear dependency of vectors (as in \cite{Akian2008, izhakian2013}). The classical definition is the following: a set of vectors is linearly dependent if for some non-trivial linear combination the maximal entry at each column is obtained at least twice. For instance, the vectors
$$w_1=(1,2,0), w_2=(0,3,2), w_3=(0,0,0)$$
are linearly dependent. Indeed,
$$(1\odot w_1) \oplus (w_2) \oplus (2\odot w_3) = (2,3,1)\oplus (0,3,2)\oplus (2,2,2) = (2,3,2).$$

We introduce an extension of the max-plus algebra with layers, called exploded layered tropical algebra (or ELT algebra for short). This structure is a generalization of the work of Izhakian and Rowen (\cite{IR1}), and is similar to Parker's exploded structure (\cite{PR}). In ELT linear algebra, a set of vectors is linearly dependent if for some non-trivial linear combination, all of the layers equal zero. For instance, the vectors
$$v_1=(\layer{1}{1},\layer{2}{1},\layer{0}{1}),v_2=(\layer{0}{1},\layer{3}{1},\layer{2}{1}),v_1=(\layer{0}{-1},\layer{0}{1},\layer{0}{1})$$
are linearly dependent. Indeed,
$$\layer{1}{1}v_1 + \layer{0}{-1}v_2 + \layer{2}{1}v_3 = (\layer{2}{0},\layer{3}{0},\layer{2}{0}).$$\\

We notice that while $u_1=(1,1), u_2=(1,1)$ are clearly linearly dependent in tropical algebra, the two vectors
$$u'_1=(\layer{1}{1},\layer{1}{-1}),u'_2=(\layer{1}{-1},\layer{1}{1})$$
are independent. Geometrically the span of these two vectors is equal to a span of one vector. However, the layers of these two spans differs. Naturally we would like to know what is the maximal size of an independent set. \\

In addition, in the literature there are several definitions of tropical ranks (see \cite{DSS, Akian2008, Akian2012, Beasley2015}). In this paper we study ELT ranks: the usual row and column ranks, the ELT submatrix rank, the ELT Kapranov rank and the ELT Barvinok rank.\\

One of the results presented in this paper is that the size of a maximal independent set is exactly~$n$ (\Cref{cor:n+1 vectors are dependent}). Furthermore, the maximal number of linearly independent rows of a matrix (called row rank) is always equal to the maximal number of independent columns.\\

We further study ELT linear algebra, and present the following main results in this paper:
\begin{enumerate}
  \item Formulation of the natural properties of linear dependence and determinant: A matrix is singular if and only if its rows are linearly dependent (Theorem \ref{minusinftydet}).
  \item The row rank and the column rank of a matrix are equal (Theorem \ref{rank}).
  \item The ELT rank and Kapranov rank of a tropical matrix are equal (Theorem \ref{eltkapranov}, Lemma \ref{eltkapranovlem}).
  \item ELT versions of Cauchy-Schwarz inequality (\Tref{thm:cauchy-schwarz}) and Bessel's inequality (\Tref{thm:bessel-inequality}).
\end{enumerate}

\subsection{ELT Algebras}

Exploded Layered Tropical algebras, or ELT algebras for short, arise from Parker's exploded semiring, which he used to study the Gromov Witten invariant (see \cite{PR}).

\begin{defn}
Let $\L$ be a semiring, and $\F$ a totally ordered semigroup. An \textbf{ELT algebra} is the pair $\R=\ELT{\F}{\L}$, whose elements are denoted $\layer a{\ell}$ for $a\in\F$ and $\ell\in\L$, together with the semiring (without zero) structure:
\begin{enumerate}
\item $\layer{a_{1}}{\ell_{1}}+\layer{a_{2}}{\ell_{2}}:=\begin{cases}
\layer{a_{1}}{\ell_{1}} & a_{1}>a_{2}\\
\layer{a_{2}}{\ell_{2}} & a_{1}<a_{2}\\
\layer{a_{1}}{\ell_{1}+_\L\ell_{2}} & a_{1}=a_{2}
\end{cases}$.
\item $\layer{a_{1}}{\ell_{1}}\cdot\layer{a_{2}}{\ell_{2}}:=\layer{\left(a_{1}+_\F a_{2}\right)}{\ell_{1}\cdot_\L\ell_{2}}$.
\end{enumerate}
We write . For $\layer{a}{\ell}$, $\ell$ is called the \textbf{layer}, whereas $a$ is called the \textbf{tangible value}.
\end{defn}

We rewrite Parker's notation $\ell\mathfrak{t}^a$ to $\layer{a}{\ell}$, in order not to be confused with $0\mathfrak{t}^a\neq 0$.\\

Let $\R$ be an ELT algebra. We write $s:\R\rightarrow\L$ for the projection on the first component (the \textbf{sorting map}):
$$s\left(\layer a{\ell}\right)=\ell$$
We also write $\t:\R\rightarrow\F$ for the projection on the second component:
$$\t\left(\layer a{\ell}\right)=a$$
We denote the \textbf{zero-layer subset}
$$\zeroset{\R}=\left\{\alpha\in \R\middle| s\left(\alpha\right)=0\right\}$$
and
$$\R^{\times}=\left\{\alpha\in\R\middle|s\left(\alpha\right)\neq 0\right\}=\R\setminus\zeroset{\R}$$\\

We note some special cases of ELT algebras.
\begin{example}
Let $\left(G,\cdot\right)$ be a totally ordered group. We denote by $G_{\max}$ the max-plus algebra defined over $G$, i.e.\ the set $G$ endowed with the operation
$$a\oplus b=\max\left\{a,b\right\},\;\;a\odot b=a\cdot b.$$
Then $G_{\max}$ is equivalent to the trivial ELT algebra with $\F=G$ and $\L=\left\{1\right\}$.
\end{example}

\begin{example}
Zur Izhakian's supertropical algebra (\cite{IZ}) is equivalent to an ELT algebra with a layering set $\L=\left\{1,\infty\right\}$, where
$$1+1=\infty,\;\;1+\infty=\infty+1=\infty,\;\;\infty+\infty=\infty$$
and
$$1\cdot 1=1,\;\;1\cdot\infty=\infty\cdot 1=\infty,\;\;\infty\cdot\infty=\infty.$$

The supertropical "ghost" elements $a^\nu$ correspond to $\layer{a}{\infty}$ in the ELT notation, whereas the tangible elements $a$ correspond to $\layer{a}{1}$.
\end{example}

We define a partial order relation $\vDash$ on $\R$ in the following way:
$$x\vDash y\Longleftrightarrow \exists z\in\zeroset{\R}:x=y+z$$

\begin{lem}
$\vDash$ is a partial order relation on $\R$.
\end{lem}
\begin{proof}
We prove that $\vDash$ is reflexive, antisymmetric and transitive.
\begin{enumerate}
  \item Reflexivity -- given $x\in\R$, we have that $x=x+\zero x$. Since $\zero x\in\zeroset{\R}$, $x\vDash x$.
  \item Antisymmetry -- suppose that $x,y\in\R$ satisfy $x\vDash y$ and $y\vDash x$. By definition, there exist $z_1,z_2\in\zeroset{\R}$ such that $x=y+z_1$ and $y=x+z_2$. We note that since $z_1\in\zeroset{\R}$, $z_1+z_1=z_1$. Therefore,
      $$x=y+z_1=y+z_1+z_1=x+z_1$$
      which proves that
      $$y=x+z_2=x+z_1+z_2=\left(x+z_2\right)+z_1=y+z_1=x$$
      as required.
  \item Transitivity -- let $x_1,x_2,x_3\in\R$ such that $x_1\vDash x_2$ and $x_2\vDash x_3$. Then there exist $z_1,z_2\in\zeroset{\R}$ such that $x_1=x_2+z_1$ and $x_2=x_3+z_2$. Therefore,
      $$x_1=x_3+\left(z_1+z_2\right).$$
      Since $z_1+z_2\in\zeroset{\R}$, we showed that $x_1\vDash x_3$.
\end{enumerate}
\end{proof}

Let us point out some important elements in any ELT algebra $\R$:
\begin{enumerate}
\item $\layer{0}{1}$, which is the multiplicative identity of $\R$.
\item $\layer{0}{0}$, which is idempotent for both operations of $\R$.
\item $\layer{0}{-1}$, which has the role of ``$-1$'' in our theory.
\end{enumerate}

Note that $\zero\cdot\layer{a}{\ell}=\layer{a}{0}$. Therefore, $\zeroset{\R}=\zero\,\R$. In particular, $\zeroset{\R}$ is an ideal of $\R$.

\subsection{The Element $-\infty$}\label{sub:the-element-minf}

As in the tropical algebra, ELT algebras lack an additive identity. Therefore, we adjoin a formal element to the ELT algebra $\R$, denoted by $-\infty$, which satisfies $\forall\alpha\in\R$:
$$\begin{array}{c}
-\infty+\alpha=\alpha+-\infty=\alpha\\
-\infty\cdot\alpha=\alpha\cdot-\infty=-\infty
\end{array}$$
We also define $s\left(-\infty\right)=0$. We denote $\overline{\R}=\R\cup\left\{-\infty\right\}$.\\

We note that $\overline{\R}$ is now a semiring, with the following property:
$$\alpha+\beta=-\infty\Longrightarrow\alpha=\beta=-\infty$$
Such a semiring is called an \textbf{antiring}. Antirings are dealt with in \cite{Tan2007, Dolzan2008}.

\begin{lem}\label{lem:surpass-minus-infty-means-zero-layer}
Let $x\in\R$. Then $s\left(x\right)=0$ if and only if $x\vDash-\infty$.
\end{lem}
\begin{proof}
On the one hand, if $s\left(x\right)=0$, then $x=\left(-\infty\right)+x$, and thus $x\vDash-\infty$.

On the other hand, if $x\vDash-\infty$, then there exists $z\in\zeroset{\overline{\R}}$ such that
$$x=\left(-\infty\right)+z=z$$
Therefore $x=z$, so $s\left(x\right)=0$.
\end{proof}

We also have the following useful lemma:
\begin{lem}\label{lem:t-equals-criterion}
Let $x,y\in\R$.
\begin{enumerate}
  \item $\t\left(x\right)\leq\t\left(y\right)$ if and only if there exists $a\in\overline{\R}$ such that $y=x+a$.
  \item $\t\left(x\right)=\t\left(y\right)$ if and only if there exist $a,b\in\overline{\R}$ such that $x=y+a$ and $y=x+b$.
\end{enumerate}
\end{lem}
\begin{proof}
We first note that the second assertion follows from the first. Thus, it suffices to prove the first assertion.

Assume $\t\left(x\right)\leq\t\left(y\right)$.
\begin{itemize}
  \item If $\t\left(x\right)<\t\left(y\right)$, take $a=y$. We get $y=x+y=x+a$.
  \item If $\t\left(x\right)=\t\left(y\right)$ and $x=y=-\infty$, one may take $a=-\infty$ to see that $y=x+a$.
  \item If $\t\left(x\right)=\t\left(y\right)$ and $x,y\neq-\infty$, write $x=\layer{\alpha}{\ell}$, $y=\layer{\alpha}{k}$, and let $a=\layer{\alpha}{k-\ell}$. Thus
      $$x+a=\layer{\alpha}{\ell} + \layer{\alpha}{k-\ell}=\layer{\alpha}{k}=y$$
\end{itemize}
In any case, we have shown the existence of $a\in\overline{\R}$ such that $y=x+a$, as required.

Now, suppose that there exists some $a\in\overline{\R}$ such that $y=x+a$. By the definition of the addition, we automatically get $\t\left(x\right)\leq\t\left(y\right)$.
\end{proof}

\subsection{Non-Archimedean Valuations and Puiseux Series}

We recall the definition of a (non-Archimedean) valuation (see \cite{Efrat2006, Tignol2015}).

\begin{defn}
Let $K$ be a field, and let $\left(\Gamma,+,\ge\right)$ be an abelian totally ordered group. Extend $\Gamma$ to $\Gamma\cup\left\{\infty\right\}$ with $\gamma<\infty$ and $\gamma+\infty=\infty+\gamma=\infty$ for all $\gamma\in\Gamma$. A function $v:K\to\Gamma\cup\left\{\infty\right\}$ is called a \textbf{valuation}, if the following properties hold:
\begin{enumerate}
  \item $v\left(x\right)=\infty\Longleftrightarrow x=0$.
  \item $\forall x,y\in K: v\left(xy\right)=v\left(x\right)+v\left(y\right)$.
  \item $\forall x,y\in K: v\left(x+y\right)\ge\min\left\{v\left(x\right), v\left(y\right)\right\}$.
\end{enumerate}
\end{defn}

Given a valuation $v$ over a field $K$, we recall some basic properties:
\begin{enumerate}
  \item $v\left(1\right)=0$.
  \item $\forall x\in K:v\left(-x\right)=v\left(x\right)$.
  \item $\forall x\in K^{\times}:v\left(x^{-1}\right)=-v\left(x\right)$.
  \item If $v\left(x+y\right)>\min\left\{v\left(x\right), v\left(y\right)\right\}$, then $v\left(x\right)=v\left(y\right)$. (For this reason, the equality between the valuation of two elements is central in out theory.)
\end{enumerate}

One may associate with $v$ the \textbf{valuation ring}
$$\mathcal{O}_v=\left\{x\in K\middle|v\left(x\right)\ge 0\right\}$$
This is a local ring with the unique maximal ideal
$$\mathfrak{m}_v=\left\{x\in K\middle|v\left(x\right)>0\right\}$$
The quotient $k_v=\quo{\mathcal{O}_v}{\mathfrak{m}_v}$ is called the \textbf{residue field} of the valuation.\\

Let us present another key construction related to valuations. For $\gamma\in\Gamma$, let $D_{\ge\gamma}=\left\{x\in K\middle|v\left(x\right)\ge\gamma\right\}$ and $D_{>\gamma}=\left\{x\in K\middle|v\left(x\right)>\gamma\right\}$. It is easily seen that $D_{\ge\gamma}$ is an abelian additive group, and that $D_{>\gamma}$ is a subgroup of $D_{\ge\gamma}$. Note that for $\gamma=0$, $D_{\ge 0}=\mathcal{O}_v$ and $D_{>0}=\mathfrak{m}_v$. Set $D_\gamma=\quo{D_{\ge\gamma}}{D_{>\gamma}}$. The \textbf{associated graded ring} of $K$ with respect to $v$ is
$$\gr_v\left(K\right)=\bigoplus_{\gamma\in\Gamma}D_\gamma$$
Given $\gamma_1,\gamma_2\in\Gamma$, the multiplication in $K$ induces a well-defined multiplication $D_{\gamma_1}\times D_{\gamma_2}\to D_{\gamma_1+\gamma_2}$ given by
$$\left(x_1+D_{>\gamma_1}\right)\cdot\left(x_2+D_{>\gamma_2}\right)=x_1x_2+D_{>\left(\gamma_1+\gamma_2\right)}$$
This multiplication can be extended to a multiplication map in $\gr_v\left(K\right)$, endowing it with a structure of a graded ring.\\

We will now focus on Puiseux series, which is the central example for our theory. The field of \textbf{Puiseux series} with coefficients in a field $K$ and exponents in an abelian ordered group $\Gamma$ is
$$K\{\{t\}\}=\left\{\sum_{i\in I}\alpha_i t^i\middle|\alpha_i\in K, I\subseteq \Gamma\text{ is well-ordered}\right\}$$
The resulting set, equipped with the natural operations, is a field; in addition, if $K$ is algebraically closed and $\Gamma$ is divisible, then $K\{\{t\}\}$ is also algebraically closed.\\

Assuming $\Gamma$ is also totally ordered, one may define a valuation on the field of Puiseux series $v:K\{\{t\}\}\to\Gamma\cup\left\{\infty\right\}$ as follows: $v\left(0\right)=\infty$, and
$$v\left(\sum_{i\in I}\alpha_i t^i\right)=\min\left\{i\in I\middle|\alpha_i\neq 0\right\}$$

Let us examine the associated graded ring with respect to this valuation. For each $\gamma\in\Gamma$, we first claim that $D_\gamma\cong K$. Indeed, the kernel of the homomorphism $f:D_{\ge\gamma}\to K$ defined by
$$f\left(\sum_{\gamma\leq i\in I}\alpha_i t^i\right)=\alpha_\gamma$$
is precisely $D_{>\gamma}$ (since $D_{>\gamma}$ is the subgroup of $D_{\ge\gamma}$ of Puiseux series whose minimal degree is bigger than $\gamma$).

\subsection{ELT Algebras and Puiseux Series}

We now consider the classical max-plus algebra $\RR_{\max}$, and suppose that $\mathbb{F}$ is some algebraically closed field. We denote by $\mathbb{K}$ the field of Puiseux series with coefficients from $\mathbb{F}$ and exponents from $\RR$, which is algebraically closed. Given a Puiseux series $x\in\mathbb{K}$, we define its \textbf{tropicalization} by
$$\Trop\left(x\right)=-v\left(x\right)$$
where $v$ is the valuation we defined on $\mathbb{K}$. This defines a function $\Trop:\mathbb{K}\to\overline{\RR_{\max}}$ (where $\overline{\RR_{\max}}=\RR_{\max}\cup\left\{-\infty\right\}$). The tropicalization function satisfies the following properties:
\begin{enumerate}
  \item $\Trop\left(x\right)\oplus\Trop\left(y\right)\ge\Trop\left(x+y\right)$, and there is equality if $v\left(x\right)\neq v\left(y\right)$.
  \item $\Trop\left(x\right)\odot\Trop\left(y\right)=\Trop\left(x\cdot y\right)$.
\end{enumerate}

A well-known result which demonstrates the connection between the max-plus algebra and Puiseux series is the Kapranov Theorem:
\begin{thm}[Kapranov Theorem]\label{thm:Kapranov}
Let $I\vartriangleleft\mathbb{K}\left[\lambda_1,\dots,\lambda_n\right]$ be an ideal of polynomials. Then
$$\Trop\left(V\left(I\right)\right)=V\left(\Trop\left(I\right)\right)$$
where:
\begin{enumerate}
  \item $V\left(I\right)$ is the set of common roots of the polynomials in $I$.
  \item $\Trop\left(X\right)$ is the tropicalization of each element in $X$.
\end{enumerate}
\end{thm}

We will now give a similar concept to the idea of tropicalization for the case of ELT algebras.\\

Considering the ELT algebra $\R=\ELT{\Gamma}{K}$, where $\Gamma$ is a totally ordered group and $K$ is a field, one may note that each $D_\gamma$ for $\gamma\in\Gamma$ corresponds to the ``section'' $\left\{\layer{\gamma}{\ell}\middle|\ell\in K\right\}$. The multiplication in $\gr_v\left(K\{\{t\}\}\right)$ is consistent with the multiplication in $\R$. The addition in $\gr_v\left(K\{\{t\}\}\right)$ is also consistent with the addition in $\R$, unless the summands have a cancellation of the minimal-degrees monomials. This is where the role of the zero-layered elements come into play -- they represent that there was a cancellation in the level of the Puiseux series.\\

Motivated by this connection between ELT algebras and Puiseux series, we further study it, generalizing it to ELT algebras for which the layering set is a ring rather than a field. We note that one can define Puiseux series where the coefficients set is a ring, and the result would be a ring of Puiseux series.\\

Let $\R=\ELT{\F}{\L}$ be an ELT algebra. Define a function $\ELTrop:\L\{\{t\}\}\rightarrow\overline{\R}$ in the following way: if $x\in \L\{\{t\}\}\backslash\left\{0\right\}$ has a leading monomial $\ell t^{a}$, then
$$\ELTrop\left(x\right)=\layer{\left(-a\right)}{\ell}.$$
In addition, $\ELTrop\left(0\right)=-\infty$.

\begin{lem}\label{lem:ELTrop-prop}
The following properties hold:
\begin{enumerate}
\item $\forall x,y\in\L\{\{t\}\}:\ELTrop\left(x\right)+\ELTrop\left(y\right)\vDash\ELTrop\left(x+y\right)$.
\item $\forall\alpha\in\L\,\forall x\in\L\{\{t\}\}:\ELTrop\left(\alpha x\right)=\layer 0{\alpha}\, \ELTrop\left(x\right)$.
\item $\forall x,y\in\L\{\{t\}\}:\ELTrop\left(x\right)\ELTrop\left(y\right)\vDash\ELTrop\left(xy\right)$.
\end{enumerate}
\end{lem}
\begin{proof}
If $x=0$ or $y=0$, the assertion is clear. Therefore, we may assume that $x,y\neq 0$, and write $\displaystyle{x=\alpha_{i_0} t^{i_0}+\sum_{i_0<i\in I}\alpha_i t^i}$ and $\displaystyle{y=\beta_{j_0} t^{j_0}+\sum_{j_0<j\in J}\beta_j t^j}$ for $\alpha_{i_0},\beta_{j_0}\neq0$.
\begin{enumerate}
\item If $i_0\neq j_0$, without loss of generality $i_0<j_0$. Then
    \begin{eqnarray*}
    \ELTrop\left(x+y\right)&=&\ELTrop\left(\alpha_{i_0} t^{i_0}+\beta_{j_0} t^{j_0}+\sum_{i_0<i\in I}\alpha_i t^i+\sum_{j_0<j\in J}\beta_j t^j\right)=\\
    &=&\layer{\left(-i_0\right)}{\alpha_{i_0}}=\layer{\left(-i_0\right)}{\alpha_{i_0}}+\layer{\left(-j_0\right)}{\beta_{j_0}}= \ELTrop\left(x\right)+\ELTrop\left(y\right)
    \end{eqnarray*}
    If $i_0=j_0$, but $\alpha_{i_0}+\beta_{j_0}\neq 0$, then
    \begin{eqnarray*}
    \ELTrop\left(x+y\right)&=&\ELTrop\left(\alpha_{i_0} t^{i_0}+\beta_{j_0} t^{j_0}+\sum_{i_0<i\in I}\alpha_i t^i+\sum_{j_0<j\in J}\beta_j t^j\right)=\\
    &=&\layer{\left(-i_0\right)}{\alpha_{i_0}+\beta_{j_0}}=\layer{\left(-i_0\right)}{\alpha_{i_0}}+\layer{\left(-j_0\right)}{\beta_{j_0}}= \ELTrop\left(x\right)+\ELTrop\left(y\right)
    \end{eqnarray*}
    We are left in the case where $i_0=j_0$ and $\alpha_{i_0}+\beta_{j_0}=0$. Thus,
    $$\ELTrop\left(x\right)+\ELTrop\left(y\right)=\layer{\left(-i_0\right)}{\alpha_{i_0}}+\layer{\left(-j_0\right)}{\beta_{j_0}}= \layer{\left(-i_0\right)}{0}$$
    Also, in this case, the leading monomials of $x$ and $y$ cancel in $x+y$. Therefore, the leading monomial of $x+y$, $\gamma_{k_0}t^{k_0}$ has degree $k_0>i_0$, and thus
    $$\ELTrop\left(x\right)+\ELTrop\left(y\right)=\layer{\left(-i_0\right)}{0}=\layer{\left(-i_0\right)}{0}+\layer{\left(-k_0\right)}{\gamma_{k_0}} \vDash\layer{\left(-k_0\right)}{\gamma_{k_0}}=\ELTrop\left(x+y\right)$$
\item This is a special case of 3, where $y$ is taken to be a constant Puiseux series.
\item First, we assume that $\alpha_{i_0}\beta_{j_0}\neq 0$. In that case, the leading monomial of $x+y$ is $\alpha_{i_0}\beta_{j_0}t^{i_0+j_0}$, and thus
    $$\ELTrop\left(xy\right)=\layer{\left(-i_0-j_0\right)}{\alpha_{i_0}\beta_{j_0}}= \layer{\left(-i_0\right)}{\alpha_{i_0}}\layer{\left(-j_0\right)}{\beta_{j_0}}=\ELTrop\left(x\right)\ELTrop\left(y\right)$$
    Otherwise, $\alpha_{i_0}\beta_{j_0}=0$. In that case, the leading monomial of $xy$, $\gamma_{k_0}t^{k_0}$ has degree $k_0>i_0+j_0$, and thus
    \begin{eqnarray*}
    \ELTrop\left(x\right)\ELTrop\left(y\right)&=&\layer{\left(-i_0\right)}{\alpha_{i_0}}\layer{\left(-j_0\right)}{\beta_{j_0}}= \layer{\left(-i_0-j_0\right)}{0}=\\
    &=&\layer{\left(-i_0-j_0\right)}{0}+\layer{\left(-k_0\right)}{\gamma_{k_0}}\vDash\layer{\left(-k_0\right)}{\gamma_{k_0}}=\ELTrop\left(xy\right)
    \end{eqnarray*}
\end{enumerate}
\end{proof}

We remark that in the case in which $\R$ is an ELT integral domain, meaning $\L$ is an integral domain, we have $\ELTrop\left(x\right)\ELTrop\left(y\right)=\ELTrop\left(xy\right)$ for all $x,y\in\L\{\{t\}\}$, since the second case in the proof, i.e.\ $\alpha_{i_0}\beta_{j_0}=0$, cannot happen. \\

Let us examine the relation $x\vDash\ELTrop\left(y\right)$ a bit more deeply. If $x=\ELTrop\left(y\right)$, it means that $x$ can be lifted to a Puiseux series which has $x$ as its leading monomial. Otherwise, we have that $x$ is of layer zero, and its tangible value is bigger than the tangible value of $\ELTrop\left(y\right)$; so one may say that $x$ can also be lifted to a Puisuex series with leading coefficient $x$, where we allow it to have a zero coefficient in its leading monomial.\\

As in the tropical case, there is a parallel theorem to Kapranov Theorem:

\begin{thm}[The Fundamental Theorem]\label{thm:Fundamental}
Let $I\vartriangleleft\mathbb{K}\left[\lambda_1,\dots,\lambda_n\right]$ be an ideal of polynomials. Then
$$\ELTrop\left(V\left(I\right)\right)=V\left(\ELTrop\left(I\right)\right)$$
where $\ELTrop\left(X\right)$ is the tropicalization of each element in $X$.
\end{thm}

A proof of this theorem may be found in \cite{Maclag2015}.

\subsection{Semirings with a Negation Map and ELT Rings}\label{sub:ELT-symmetrized}

Semirings need not have additive inverses to all of the elements. While some of the theory of rings can be copied ``as-is'' to semirings, there are many facts about rings which use the additive inverses of the elements. The idea of negation maps on semirings (sometimes called symmetrized semirings) is to imitate the additive inverse map. Semirings with negation maps are discussed in \cite{Akian1990, Gaubert1992, Gaubert1997, Akian2008, Akian2014, Rowen2016}.

\begin{defn}
Let $R$ be a semiring. A map $(-):R\to R$ is a \textbf{negation map} (or a \textbf{symmetry})
if the following properties hold:
\begin{enumerate}
\item $\forall a,b\in R:(-)\left(a+b\right)=(-)a+(-)b$.
\item $(-)0_R=0_R$.
\item $\forall a,b\in R:(-)\left(a\cdot b\right)=a\cdot\left((-)b\right)=\left((-)a\right)\cdot b$.
\item $\forall a\in R:(-)\left((-)a\right)=a$.
\end{enumerate}
We say that $\left(R,(-)\right)$ is a \textbf{semiring with a negation map} (or a \textbf{symmetrized semiring}). If $(-)$ is clear from the context, we will not mention it.
\end{defn}

We give several examples of semirings with negation maps:
\begin{itemize}
\item A trivial example of a negation map (over any semiring) is $(-)a=a$.
\item If $R$ is a ring, it has a negation map $(-)a=-a$.
\item If $\R$ is an ELT algebra, we have a negation map given by $(-)a=\minus a$.
\end{itemize}

The last example is the central example for our theory, since it shows that any ELT algebra is equipped with a natural negation map. Thus, the theory of semirings with negation maps can be used when dealing with ELT algebras.\\

We now present several notations from this theory:
\begin{itemize}
\item $a+(-)a$ is denoted $a^\circ$.
\item $R^\circ=\left\{a^\circ\middle|a\in R\right\}$.
\item We define two partial orders on $R$:
\begin{itemize}
\item The relation $\symmdash$ defined by
$$a\symmdash b\Leftrightarrow \exists c\in R^\circ:a=b+c$$
\item The relation $\nabla$ defined by
$$a\nabla b\Leftrightarrow a+(-)b\in R^\circ$$
\end{itemize}
\end{itemize}

If $\R$ is an ELT algebra, then some of these notations have already been defined. For example, $a^\circ=\zero a$, $\R^\circ=\zeroset{\R}$ and the relation $\symmdash$ is the relation $\vDash$.

\section{ELT Matrices}
Throughout this section, we fix our ELT algebras to be of the form $\R=\ELT{\mathbb{R}}{\mathbb{F}}$, where $\mathbb{F}$ is an algebraically closed field. We also denote $\mathbb{K}=\mathbb{F}\{\{t\}\}$ the field of Puiseux series with coefficients in~$\mathbb{F}$ and powers in $\RR$.\\

Before delving into the theory, we extend some of our definitions for ELT algebras to ELT matrices.

\begin{defn}
Let $A\in\left(\overline{\R}\right)^{n\times n}$ be an ELT matrix, $A=\left(a_{i,j}\right)$. We say that its \textbf{layer} is
$$s\left(A\right)=\left(s\left(a_{i,j}\right)\right)\in\left(\L\{\{t\}\}\right)^{n\times n}.$$
If $s\left(A\right)=0$, we say that $A$ is of \textbf{layer zero}.\\

We also define a surpassing relation $\vDash$ on $\left(\overline{\R}\right)^{n\times n}$ similarly to the case of ELT algebras:
$$A\vDash B\Longleftrightarrow\exists C\in\left(\overline{\R}\right)^{n\times n},s\left(C\right)=0:A=B+C.$$
As in the case of ELT algebras, $\vDash$ is a partial order relation on $\left(\overline{\R}\right)^{n\times n}$.
\end{defn}

\begin{example}
Given $A=\left(\begin{matrix}\layer{1}{2}&\layer{3}{0}\\\layer{\left(-1\right)}{0}&\layer{\left(-3\right)}{-3}\end{matrix}\right)$, $B=\left(\begin{matrix}\layer{1}{2}&\layer{5}{0}\\-\infty&\layer{\left(-3\right)}{-3}\end{matrix}\right)$ and $C=\left(\begin{matrix}\layer{0}{0}&\layer{3}{0}\\\layer{\left(-2\right)}{0}&\layer{\left(-3\right)}{-3}\end{matrix}\right)$, we have $A\vDash B$ but $A\nvDash C$ (since $\layer{1}{2}\nvDash\layer{0}{0}$).
\end{example}

\subsection{The Exploded-Layered Tropical Determinant and Linear Dependence}

\begin{defn}
Let $\R$ be an ELT algebra, and let $A=(a_{i,j})\in \left(\overline{\R}\right)^{n\times n}$ be an ELT matrix. The \textbf{ELT determinant} of $A$ is
$$\det A=\sum_{\sigma\in S_n} {\layer{0}{\sign(\sigma)}}\cdot a_{1,\sigma(1)}\cdot\dots\cdot a_{n,\sigma(n)}.$$

We call a matrix $A\in \left(\overline{\R}\right)^{n\times n}$ \textbf{singular} if $s\left(\det A\right)=0$.
\end{defn}

\begin{lem}\label{lem:det-of-surpassing}
If $A\vDash B$, then $\det A\vDash\det B$.
\end{lem}
\begin{proof}
Write $A=B+C$, where $s\left(C\right)=0$. By expanding the ELT determinant we see that
\begin{eqnarray*}
  \det A &=&\sum_{\sigma\in S_n} {\layer{0}{\sign(\sigma)}}\cdot a_{1,\sigma(1)}\cdot\dots\cdot a_{n,\sigma(n)}=\\
   &=&\sum_{\sigma\in S_n}{\layer{0}{\sign(\sigma)}}\cdot\left(b_{1,\sigma(1)}+c_{1,\sigma(1)}\right)\cdot\dots\cdots \left(b_{n,\sigma(n)}+c_{n,\sigma(n)}\right)=\\
   &=&\sum_{\sigma\in S_n} {\layer{0}{\sign(\sigma)}}\cdot b_{1,\sigma(1)}\cdot\dots\cdot b_{n,\sigma(n)}+\textrm{other summands with elements of }C\textrm{ in them}=\\
   &=&\det B+\textrm{other summands with elements of }C\textrm{ in them}.
\end{eqnarray*}
Since $s\left(C\right)=0$, all of its elements are of layer zero. Hence, the summands in the RHS other than $\det B$ are also of layer zero. In conclusion, $\det A\vDash\det B$.
\end{proof}

Recall that $\R^{\times}$ is the subset of $\R$ containing the non-zero layered elements
$$\R^{\times}:=\{\layer{a}{\ell}\in \R|\ell\neq 0_\mathbb{F}\}.$$
and $\overline{\R^{\times}}=\R^{\times}\cup\{-\infty\}$.

\begin{defn}
A set of vectors $S\subseteq \left(\overline{\R}\right)^n$ is called \textbf{linearly dependent} if there exist $v_1,...,v_m\in S$ and
$$a_1,...,a_m\in \R^{\times}$$
such that
$$s\Big(\sum_{i=1}^m a_iv_i\Big)=(0_\mathbb{F},...,0_\mathbb{F}).$$
\end{defn}

In this remainder of this subsection we prove that a matrix is singular if and only if its rows and columns are linearly dependent. Our main theorem, which we will prove later, is:

\begin{thm}\label{minusinftydet}
Consider $A\in\left(\overline{\R}\right)^{n\times n}$. Then the rows of $A$ are linearly dependent, iff the columns of $A$ are linearly dependent, iff $s\left(\det A\right)=0_\mathbb{F}$.
\end{thm}

The idea of the proof is to use the Fundamental Theorem. There is a subtle point here: the image of the EL tropicalization is the non-zero layered elements in $\R$ (and $\minf$); therefore, if our matrix contains an element of layer zero (other than $\minf$), it has no lift. We solve this problem by showing that every singular matrix (respectively, a matrix with linearly dependent rows) surpasses a singular matrix (respectively, a matrix with linearly dependent rows) with elements in $\overline{\R^{\times}}$.

\begin{lem}\label{lem:singular_surpasses_singular}
Let $A\in\left(\overline{\R}\right)^{n\times n}$ be a singular ELT matrix. Then there exists a singular ELT matrix $B\in\left(\overline{\R^\times}\right)^{n\times n}$ such that $A\vDash B$.
\end{lem}
\begin{proof}
We denote by $\mathrm{track}_\sigma\left(A\right)$ the track of the permutation $\sigma\in S_n$; that is,
$$\mathrm{track}_\sigma\left(A\right)=a_{1,\sigma\left(1\right)}\dots a_{n,\sigma\left(n\right)}$$
We also denote $A=\left(a_{i,j}\right)$. $A$ is singular, so $s\left(\det A\right)=0$.\\

We divide our proof to three cases:
\begin{enumerate}
  \item $\det A=\minf$.
  \item $\det A$ is achieved from a cancellation of (at least) two dominant non-zero layered tracks in~$A$.
  \item $\det A$ is achieved from a dominant zero-layered track in $A$.
\end{enumerate}
~
\begin{casenv}
\item Assume $\det A=\minf$. This means that the value of each track in $A$ is $\minf$ (since ELT algebras are antirings), and thus each track contains $\minf$ as one of the elements in the product (since we are dealing with an ELT field). Consider the matrix $B$ defined as follows:
    $$\left(B\right)_{i,j}=\left\{\begin{matrix}a_{i,j},&s\left(a_{i,j}\right)\neq 0\\\minf,&s\left(a_{i,j}\right)=0\end{matrix}\right.$$
    We claim that $A\vDash B$. Indeed, we construct a matrix $C\in\left(\overline{\R}\right)^{n\times n}$ as follows:
    $$\left(C\right)_{i,j}=\left\{\begin{matrix}\minf,&s\left(a_{i,j}\right)\neq 0\\a_{i,j},&s\left(a_{i,j}\right)=0\end{matrix}\right.$$
    Then $s\left(C\right)=0$ by its construction, and also
    $$\left(B+C\right)_{i,j}=\left\{\begin{matrix}a_{i,j}+\left(\minf\right),&s\left(a_{i,j}\right)\neq 0\\\left(\minf\right)+a_{i,j},&s\left(a_{i,j}\right)=0\end{matrix}\right.=a_{i,j}$$
    Hence $A=B+C$, and thus $A\vDash B$.\\
    $\det A=\minf$, which means that in every track of $A$ there is an element which equals to~$\minf$. Since $A\vDash B$, each track in $B$ will also contain $\minf$ (since if $\minf\vDash\alpha$, we must have $\alpha=\minf$). So $\det B=\minf$, and thus $B$ is singular.
\item Assume $\det A$ is achieved from a cancellation of (at least) two dominant non-zero layered tracks in $A$. Define the same matrix $B$ as in the first case. Again, $A\vDash B$. Note that we did not affect the values of the dominant track, but we ``erased'' the zero-layered tracks. Since no zero-layered track was essential in the determinant of $A$, we shall have $\det A=\det B$, hence $B$ is also singular.\\
\item We are left with the case where $\det A$ is achieved from a dominant zero-layered track in~$A$. We set $X_1$ to be the set of permutations in $S_n$ whose tracks in $A$ are non-zero layered, that is
    $$X_1=\left\{\pi\in S_n\mid s\left(\mathrm{track}_\pi\left(A\right)\right)\neq 0\right\}$$
    We also set $X_2$ to be the set of permutations in $S_n$ whose tracks are zero-layered and dominate over the non-zero layered tracks, that is
    $$X_2=\left\{\sigma\in S_n\mid s\left(\mathrm{track}_\pi\left(A\right)\right)=0, \forall\pi\in X_1:\t\left(\mathrm{track}_\sigma\left(A\right)\right)>\t\left(\mathrm{track}_\pi\left(A\right)\right)\right\}$$
    By our assumption, $X_2\neq\varnothing$.\\

    Let $a_{i_1,j_1},\dots,a_{i_m,j_m}$ be the zero-layered elements in $A$ different than $\minf$. For the moment, we replace each zero-layered element in $A$ with a variable $\lambda_1,\dots,\lambda_m$. Denote the matrix $A$ after this replacement by $A_{\lambda_1,\dots,\lambda_m}$. Then the track of each $\sigma\in S_n$ is a monomial in the variables $\lambda_1,\dots,\lambda_m$.\\

    We first set the tangible values of $\lambda_1,\dots,\lambda_m$. Let $\beta\in\R$ be the sum of the non-zero layered tracks in $A$. If $s\left(\beta\right)=0$, we may substitute $\lambda_1=\cdots=\lambda_m=\minf$, and we are done. So we assume that $\beta\in\R^{\times}$. Consider the function
    $$f:\left(\minf,\t\left(a_{i_1,j_1}\right)\right]\times\dots\times\left(\minf,\t\left(a_{i_m,j_m}\right)\right]\to\mathbb{R}$$
    given by
    $$f\left(x_1,\dots,x_m\right)=\max_{\sigma\in X_2}\t\left(\mathrm{track}_\sigma\left(A_{\layer{x_1}{1},\dots,\layer{x_m}{1}}\right)\right).$$
    One should note that we consider only the dominant zero-layered tracks in $A$ (which are, by definition, not $\minf$). The value of $f$ does not depend on the values of the layers of $x_1,\dots,x_m$, which are $1$ in the definition of $f$. We also note that by expanding $f$, one may see that $f$ is in fact a tropical polynomial, i.e.\ the maximum between several linear functions in the variables $x_1,\dots,x_m$.\\

    As $\left(x_1,\dots,x_m\right)$ tends to $\left(\minf,\dots,\minf\right)$, $f\left(x_1,\dots,x_m\right)$ tends to~$\minf$, and also
    $$f\left(\t\left(a_{i_1,j_1}\right),\dots,\t\left(a_{i_m,j_m}\right)\right)>\t\left(\beta\right).$$
    Since $f$ is continuous, there is a point $\left(x_1,\dots,x_m\right)\in\left(\minf,\t\left(a_{i_1,j_1}\right)\right)\times\dots\times\left(\minf,\t\left(a_{i_m,j_m}\right)\right)$ such that $f\left(x_1,\dots,x_m\right)=\t\left(\beta\right)$. We write $\lambda_1=\layer{x_1}{\ell_1},\dots,\lambda_m=\layer{x_m}{\ell_m}$, where $\ell_1,\dots,\ell_m$ are variables.\\

    Now, $s\left(\det A_{\lambda_1,\dots,\lambda_m}\right)$ is a non-constant polynomial expression in the variables $\ell_1,\dots,\ell_m$, which we will denote $p\left(\ell_1,\dots,\ell_m\right)$. We want to find a root for this polynomial. We note that every monomial which appears in $p$ is of the form $\beta\ell_{k_1}\dots\ell_{k_s}$.\\

    Take a monomial with a minimal number of variables appearing in it; without loss of generality, we assume it is $\beta\ell_1\dots\ell_k$. We set $\ell_r=0$ for every $k<r\leq m$; this means that $p\left(\ell_1,\dots,\ell_k,0,\dots,0\right)=\beta\ell_1\dots\ell_k+\gamma$. Now it is easy to pick appropriate values for $\ell_1,\dots,\ell_k$ so that $p\left(\ell_1,\dots,\ell_k,0,\dots,0\right)=0$.\\

    We take $B=A_{\layer{x_1}{\ell_1},\dots,\layer{x_k}{\ell_k},\minf,\dots,\minf}$. We remark that although $x_{k+1},\dots,x_m\neq\minf$ by their construction, we chose to set them to $\minf$, so they will not affect the determinant. It is easily seen that $A\vDash B$, and by its construction $B\in\left(\overline{\R^\times}\right)^{n\times n}$. To see why $B$ is singular, one should observe that the dominant tracks in the determinant of $B$ are the non-zero layered tracks (which were essential in~$A$), whose sum is $\layer{b}{\gamma}$, and $\layer{b}{\beta}\layer{x_1}{\ell_1}\dots\layer{x_k}{\ell_k}$, and the sum of these tracks is zero-layered.
\end{casenv}

To conclude, we found a matrix $B\in\left(\overline{\R^\times}\right)^{n\times n}$ such that $B$ is singular and $A\vDash B$, as required.
\end{proof}

\begin{lem}\label{lem:rows_dependent_surpasses_dependent}
Let $A\in \left(\overline{\R}\right)^{n\times n}$ be an ELT matrix with linearly dependent rows. Then there exists an ELT matrix $B\in \left(\overline{\R^\times}\right)^{n\times n}$ such that $A\vDash B$ and the rows of $B$ are linearly dependent.
\end{lem}
\begin{proof}
Write $A=\left(a_{i,j}\right)$, and let $\alpha_1,\dots,\alpha_n\in\overline{\R^{\times}}$, not all are $\minf$, such that $s\left(\displaystyle{\sum_{i=1}^{n}\alpha_i R_i\left(A\right)}\right)=0$. We construct the matrix $B$ as follows: if $\left(\displaystyle{\sum_{i=1}^{n}\alpha_i R_i\left(A\right)}\right)_j$ is of layer, yet the dominant elements in this sum are not zero-layered, we define
$$\left(B\right)_{i,j}=\left\{\begin{matrix}a_{i,j},&s\left(a_{i,j}\right)\neq0\\\minf,&s\left(a_{i,j}\right)=0\end{matrix}\right.$$

Otherwise, there is some zero-layered element $a_{k,j}$ such that $\alpha_k a_{k,j}$ dominates in $\left(\displaystyle{\sum_{i=1}^{n}\alpha_i R_i\left(A\right)}\right)_j$. So we define
$$b_{i,j}=\left(B\right)_{i,j}=\left\{\begin{matrix}a_{i,j},&s\left(a_{i,j}\right)\neq 0\\\minf,&s\left(a_{i,j}\right)=0,i\neq k\\\layer{0}{-1}\,\alpha_k^{-1}\displaystyle{\sum_{\substack{\ell=1\\s\left(a_{i,\ell}\right)\neq 0}}^{n}\alpha_\ell a_{\ell,j}},&i=k\end{matrix}\right.$$
(If the last sum is of layer zero, we define $b_{k,j}=\minf$.)\\

We note that $C_j\left(B\right)\vDash C_j\left(A\right)$. Indeed, for $i\neq k$, $b_{i,j}$ is either $a_{i,j}$ or $-\infty$ if $s\left(a_{i,j}\right)=0$. Hence we are only left with the case when $i=k$. Since $\alpha_k a_{k,j}$ dominates in $\left(\displaystyle{\sum_{i=1}^{n}\alpha_i R_i\left(A\right)}\right)_j$,
$$\t\left(\alpha_k a_{k,j}\right)\ge\t\left(\sum_{\substack{i=1\\i\neq k}}^{n}\alpha_i a_{i,k}\right)>\t\left(\sum_{\substack{i=1\\s\left(a_{i,\ell}\neq 0\right)}}^{n}\alpha_i a_{i,k}\right)$$
where the last inequality holds since the dominant elements in the sum $\t\left(\sum_{\substack{i=1\\i\neq k}}^{n}\alpha_i a_{i,k}\right)$ are not zero-layered. Hence
$$\t\left(a_{k,j}\right)>\t\left(\alpha_k^{-1}\sum_{\substack{i=1\\s\left(a_{i,\ell}\neq 0\right)}}^{n}\alpha_i a_{i,k}\right)=\t\left(b_{k,j}\right).$$
Since $s\left(a_{k,j}\right)=0$, we get that $a_{k,j}\vDash b_{k,j}$.

By its construction, $A\vDash B$ (since its column of $A$ surpasses the appropriate column of $B$); also, $B$ satisfies $s\left(\displaystyle{\sum_{i=1}^{n}\alpha_i R_i\left(B\right)}\right)=0$, as required.
\end{proof}

\begin{lem}\label{lem:lifting matrix with dependent rows}
Let $A\in \left(\overline{\R^\times}\right)^{n\times n}$ be an ELT matrix whose rows are linearly dependent. Then there is a matrix $B\in \mathbb{K}^{n\times n}$ such that $\ELTrop\left[B\right]=A$ and the rows of $B$ are linearly dependent.
\end{lem}
\begin{proof}
Since the rows of $A$ are linearly dependent, there are scalars $\alpha_1,\dots,\alpha_n\in\overline{\R^\times}$, not all are equal to~$\minf$, such that
$$s\left(\sum_{i=1}^{n}\alpha_i R_i\left(A\right)\right)=\left(0,0,\dots,0\right)$$
Consider the polynomial $g\left(\lambda_1,\dots,\lambda_n\right)=\displaystyle{\sum_{i=1}^{n}\alpha_i\lambda_i}$, and fix some polynomial $f\in \mathbb{K}\left[\lambda_1,\dots,\lambda_n\right]$ such that $\ELTrop\left[f\right]=g$.

We now lift every column separately. Each column of $A$ can be considered as a point in $\left(\overline{R^\times}\right)^n$ which is an ELT root of $g$; hence, by the Fundamental Theorem (\Tref{thm:Fundamental}), each column $C_j\left(A\right)$ has a lift $v_j\in \mathbb{K}^n$ such that $\ELTrop\left[v_j\right]=C_j\left(A\right)$ and $f\left(v_j\right)=0$. Hence, the matrix $B$ whose columns are $v_j$ satisfies $\ELTrop\left[B\right]=A$, and its rows are linearly dependent (by the choice of $f$).
\end{proof}

Before proving the main theorem, we need to remark about EL tropicalization of matrices. Using the function $\ELTrop$ defined in the introduction, one may define a similar function for matrices of Puiseux series, $\ELTrop:\mathbb{K}^{n\times n}\to \left(\overline{\R}\right)^{n\times n}$, by
$$\ELTrop\left[\left(a_{i,j}\right)\right]=\left(\ELTrop\left(a_{i,j}\right)\right)$$
Using \Lref{lem:ELTrop-prop}, one can easily prove that:
\begin{enumerate}
  \item $\ELTrop\left[A\right]+\ELTrop\left[B\right]\vDash\ELTrop\left[A+B\right]$.
  \item $\ELTrop\left(\alpha\right)\ELTrop\left[A\right]\vDash\ELTrop\left[\alpha A\right]$.
  \item $\ELTrop\left[A\right]\cdot\ELTrop\left[B\right]\vDash\ELTrop\left[AB\right]$.
  \item $\det\ELTrop\left[A\right]\vDash\ELTrop\left(\det A\right)$.
\end{enumerate}

We shall now use these facts and some earlier lemmas to prove the main theorem.

\begin{proof}[Proof of \Tref{minusinftydet}]
First, assume that the rows of $A$ are linearly dependent. By \Lref{lem:rows_dependent_surpasses_dependent}, there is a matrix $B\in \left(\overline{\R^{\times}}\right)^{n\times n}$ such that $A\vDash B$ and the rows of $B$ are linearly dependent. By \Lref{lem:lifting matrix with dependent rows}, there exists a lift $\tilde{B}\in \mathbb{K}^{n\times n}$ such that $\ELTrop\left[\tilde{B}\right]=B$ and the rows of $\tilde{B}$ are linearly dependent. But, since we are now working with a matrix over a field, $\det\tilde{B}=0$. Hence
$$\det A\overset{\Lref{lem:det-of-surpassing}}{\vDash}\det B=\det\ELTrop\left[\tilde{B}\right]\vDash\ELTrop\left(\det\tilde{B}\right)=\minf$$
Therefore, by \Lref{lem:surpass-minus-infty-means-zero-layer}, $s\left(\det A\right)=0$, i.e.\ $A$ is singular.\\

Now, assume that $A$ is singular. By \Lref{lem:singular_surpasses_singular}, there is a matrix $B\in \left(\overline{\R^{\times}}\right)^{n\times n}$ such that $A\vDash B$ and $B$ is singular. By the Fundamental Theorem, there exists a lift $\tilde{B}\in \mathbb{K}^{n\times n}$ such that $\ELTrop\left[\tilde{B}\right]=B$ with $\tilde{B}$ being singular, and thus its rows are linearly dependent. If
$$\sum_{j=1}^k \alpha_j R_{i_j}\left(\tilde{B}\right)=0$$
with all $\alpha_j\neq 0$, then
\begin{eqnarray*}
  \sum_{j=1}^k \ELTrop\left[\alpha_j\right] R_{i_j}\left(A\right) &\vDash& \sum_{j=1}^k \ELTrop\left[\alpha_j\right] R_{i_j}\left(B\right)=\sum_{j=1}^k \ELTrop\left[\alpha_j\right]\ELTrop\left(R_{i_j}\left(\tilde{B}\right)\right)\vDash\\
   &\vDash&\sum_{j=1}^k \ELTrop\left(\alpha_j R_{i_j}\left(\tilde{B}\right)\right)\vDash\ELTrop\left(\sum_{j=1}^k\alpha_j R_{i_j}\left(\tilde{B}\right)\right)=\left(\minf,\dots,\minf\right)
\end{eqnarray*}
Hence $s\left(\sum_{j=1}^k \ELTrop\left[\alpha_j\right] R_{i_j}\left(A\right)\right)=\left(0,\dots,0\right)$, i.e.\ the rows of $A$ are linearly dependent.
\end{proof}

\subsection{Invertible Matrices}

For this subsection only, we allow our ELT algebras to be of the form $\ELT{\F}{\L}$ for a totally ordered commutative group $\F$ and a commutative ring (with unit) $\L$.

\begin{defn}
Let $\R$ be an ELT algebra. A matrix $A\in\left(\overline{\R}\right)^{n\times n}$ is said to be \textbf{invertible} if there exists $B\in\left(\overline{\R}\right)^{n\times n}$ such that $AB=BA=I_{n}$.
\end{defn}

We will now try to find all of the left invertible matrices.

\begin{defn}
Let $\R$ be an ELT algebra. A \textbf{generalized permutation matrix} is a matrix of the form
$$\begin{pmatrix}\\
c_{1}\cdot e_{\sigma\left(1\right)} & \cdots & c_{n}\cdot e_{\sigma\left(n\right)}\\
\\
\end{pmatrix}$$
for invertible $c_{i}\in\R$ and $\sigma\in S_{n}$. If each $c_{i}=\one$, we denote that matrix $P_{\sigma}$.
\end{defn}

\begin{rem}
Any generalized permutation matrix can be written as a product of a diagonal matrix with a permutation matrix. Specifically,
$$\begin{pmatrix}\\c_1\cdot e_{\sigma\left(1\right)} & \cdots & c_n\cdot e_{\sigma\left(n\right)}\\\\
\end{pmatrix}=\begin{pmatrix}c_{\sigma\left(1\right)} &  & -\infty\\&\ddots\\-\infty&&c_{\sigma\left(n\right)}\end{pmatrix}P_\sigma$$
\end{rem}

The following theorem is a special case of \cite[Theorem 1]{Dolzan2008}, combined with \cite{Reutenauer1982}.
\begin{thm}\label{thm:invertible-matrices}
If $B\in \left(\overline{\R}\right)^{n\times n}$ is left invertible, then it is a generalized permutation matrix.
\end{thm}
\begin{proof}
By \cite{Reutenauer1982}, since $\R$ is commutative, $B$ is invertible (and not only left invertible).

Now, according to \cite[Theorem 1]{Dolzan2008}, there exists an invertible diagonal matrix $D$ and $a_{\sigma}\in\overline{\R}$ such that
$$B=D\sum_{\sigma\in S_{n}}a_{\sigma}P_{\sigma}$$
where ${\displaystyle \sum_{\sigma\in S_{n}}}a_{\sigma}=\one$ and $a_{\sigma}a_{\tau}=-\infty$ if $\sigma\neq\tau$. But that condition yields that only one $a_{\sigma}\neq-\infty$, and thus $a_{\sigma}=\one$, implying $B=D\cdot P_{\sigma}$.
\end{proof}

\subsection{Rank of a matrix}

In this section, we generalize Theorem \ref{minusinftydet} and prove that the ELT row rank of an ELT matrix is equal to the ELT column rank. We recall our assumption that the field of layers $\mathbb{F}$ is algebraically closed, and our notation $\mathbb{K}=\mathbb{F}\{\{t\}\}$ the field of Puiseux series with coefficients in $\mathbb{F}$ and powers in $\mathbb{R}$.\\

\begin{defn}
Let $A\in (\overline{\R})^{m\times n}$ be an ELT matrix. The maximal number of linearly independent rows from $A$, is called the \textbf{ELT row rank} of $A$.\\

Similarly, the \textbf{ELT column rank} of $A$ is the maximal number of linearly independent columns of $A$.\\
\end{defn}

\begin{defn}\label{tropicalrank}
Let $A\in (\overline{\R})^{m\times n}$ be an ELT matrix. The \textbf{ELT submatrix rank} of $A$ is the maximal size of a square nonsingular submatrix of $A$. If no such matrix exists, then the ELT submatrix rank of $A$ is defined to be zero.\\
\end{defn}

We aim to prove that these three definitions of rank coincide (\Tref{rank}).

\subsubsection{Kapranov and Barvinok rank}

In their paper \cite{DSS}, Develin, Santos and Sturmfels review three different definitions of matrix rank: Barvinok, Kapranov and tropical rank. Furthermore, they prove that for any tropical matrix $A$
$$\text{tropical-rank}(A) \leq \Kapranovrank(A)\leq \text{Barvinok-rank}(A),$$
where both of these inequalities can be strict.
We will present analogous definitions for rank admitting the above inequalities.\\

The analog of tropical rank is the ELT submatrix rank we introduced in definition~\ref{tropicalrank}.\\

\begin{defn}\label{kapranov_rank_def}
Let $A\in (\overline{\R^{\times}})^{m\times n}$ be an ELT matrix. The \textbf{ELT Kapranov rank} of $A$ is the minimal rank of any matrix $A(t)\in \mathbb{K}^{m\times n}$ such that $\ELTrop(A(t))=A$.\\
\end{defn}

\begin{defn}
Let $A\in (\overline{\R^{\times}})^{m\times n}$ be an ELT matrix. The \textbf{ELT Barvinok rank} of $A$ is the minimal number $r$ of matrices $A_1,...,A_r$ of submatrix rank 1, such that $A_1+A_2+...+A_r=A$.\\
\end{defn}

Since linear dependence of vectors in $\mathbb{K}^n$ implies linear dependence of their tropicalization, one can easily verify that
$$\ELTsubmatrixrank(A)\leq \ELTkapranovrank(A)\leq \text{ELT-Barvinok-rank}(A).$$

\begin{lem}\label{same_rank_lift}
For every ELT matrix $A\in (\overline{\R^{\times}})^{m\times n}$, there exists a matrix $A(t)\in\mathbb{K}^{m\times n}$ such that $\ELTrop\Big(A(t)\Big)=A$ and
$$\ELTsubmatrixrank(A)=\rank\Big(A(t)\Big).$$
\end{lem}
\begin{proof}

Write $\ELTsubmatrixrank(A)=r$. There exists a submatrix of $A$ of size $r\times r$ that is nonsingular, and every larger submatrix is singular.\\

Let $G_r$ be the set of generators of the classical determinantal ideal of size $r$, over $\mathbb{K}^{m\times n}$. For every polynomial $g\in G_{r+1}$, the matrix $A$ is an ELT root of $\ELTrop[g]$.\\

Since $G_{r+1}$ is a Gr\"obner basis of the determinantal ideal $I_{r+1}$ (ref. \cite{DSS}),  $A$ is an ELT root of $\ELTrop[f]$ for every $f\in I_{r+1}$. By the fundamental theorem (ref. \cite{Maclag2015}), there exists a matrix $A(t)\in V(I_{r+1})\subseteq \mathbb{K}^{m\times n}$ such that $\ELTrop(A(t))=A$.\\

Now $rank\Big(A(t)\Big)\leq r$, since $A(t)\in V(I_{r+1})$. Also
$$\rank\Big(A(t)\Big)\geq \ELTsubmatrixrank(\ELTrop[A(t)])=\ELTsubmatrixrank(A)=r.$$
Together,
$$\rank\Big(A(t)\Big)=\ELTsubmatrixrank(A).$$
\end{proof}

\begin{thm}\label{eltkapranov}
For any ELT matrix $A\in (\overline{\R^{\times}})^{m\times n}$
$$\ELTsubmatrixrank(A)= \ELTkapranovrank(A)$$
\end{thm}

\begin{proof}
We need to prove the inequality
$$\ELTsubmatrixrank(A)\geq \ELTkapranovrank(A).$$

By \Lref{same_rank_lift} there exist a matrix $A(t)$ such that $\ELTrop\Big(A(t)\Big)=A$ and
$$\ELTsubmatrixrank(A)=\rank\Big(A(t)\Big).$$

By Definition \ref{kapranov_rank_def},
$$\rank\Big(A(t)\Big)\geq \ELTkapranovrank(A),$$
and thus
$$\ELTsubmatrixrank(A)\geq \ELTkapranovrank(A).$$
\end{proof}

We are now ready to prove the rank theorem.

\begin{thm}[Rank Theorem]\label{rank}
Let $A\in (\overline{\R})^{m\times n}$ be an ELT matrix. Then the ELT row rank of $A$ is equal to the ELT column rank of $A$ and to the ELT submatrix rank of $A$.
\end{thm}
\begin{proof}
For convenience, write $k_r$ for the ELT row rank of $A$, and
$$k_s=\ELTsubmatrixrank(A)=\ELTkapranovrank(A).$$
We will prove that $k_r\ge k_s$ and $k_r\leq k_s$, thus proving the desired equality.\\

First we prove that $k_r\ge k_s$. If $k_r=m$, we are done, since $k_s\leq m$. So we assume that $k_r<m$. By the definition of $k_r$, every $k_r+1$ rows of $A$ are linearly dependent. Thus, by \Tref{minusinftydet}, every $\left(k_r+1\right)\times\left(k_r+1\right)$ submatrix of $A$ is singular. Therefore, by the definition of $k_s$, $k_r\ge k_s$.\\

Now we prove that $k_r\leq k_s$. Since $k_s=\ELTkapranovrank(A)$, there exists a lift $\tilde{A}\in \mathbb{K}^{m\times n}$ such that $\rank\left(\tilde{A}\right)=k_s$. Thus, every $k_s+1$ rows of $\tilde{A}$ are linearly dependent. Since linear dependence of a lift implies linear dependence in the ELT context, we get that every $k_s+1$ rows of $A$ are linearly dependent, as desired.\\

This proves that for every ELT matrix $A$ without zero-layered elements, its ELT row rank equals to its ELT submatrix rank. Since $\ELTsubmatrixrank(A)=\ELTsubmatrixrank(A^t)$, they are also equal to the ELT column rank of $A$.
\end{proof}

As an immediate corollary, we get the following:

\begin{cor}\label{cor:n+1 vectors are dependent}
Any $n+1$ vectors in $\left(\overline{\R}\right)^n$ are linearly dependent.
\end{cor}
\begin{proof}
Let $v_1,\dots,v_{n + 1}\in\left(\overline{\R}\right)^n$ be vectors. Consider the ELT matrix $A\in \left(\overline{\R}\right)^{n\times\left(n+1\right)}$ whose columns are $v_i$. Since its row rank equals to its column rank, and its row rank is bounded by $n$, we get that the column rank of $A$ is at most $n$; hence, its columns are linearly dependent. So $v_1,\dots,v_{n+1}$ are linearly dependent.
\end{proof}

\subsubsection{The ELT Rank of a Tropical Matrix}

Next, we define the ELT rank of a tropical matrix, i.e.\ a matrix over the tropical semifield $\mathbb{T}=\overline{\mathbb{R}_{\max}}$.\\

\begin{defn}
Let $A\in\mathbb{T}^{m\times n}$ be a tropical matrix. The \textbf{ELT rank} of $A$ is the minimal ELT submatrix rank of any matrix $E_A\in \big(\overline{\R^{\times}}\big)^{m\times n}$ such that $\tau(E_A)=A$ (in other words, $E_A$ is obtained by assigning layers to the entries of $A$).\\
\end{defn}

\begin{example}
Consider $A\in\mathbb{T}^{m\times n}$,
$$A=\begin{pmatrix}0 & -\infty \\ -\infty & 0 \end{pmatrix}.$$
Any matrix $E_A\in \big(\overline{\R^{\times}}\big)^{m\times n}$ such that $\tau(E_A)=A$ is of the form
$$E_A=\begin{pmatrix}\layer{0}{x_1} & -\infty \\ -\infty & \layer{0}{x_2} \end{pmatrix},$$
where $x_1,x_2\neq 0$.\\

Now for every choice of $E_A$
$$\ELTsubmatrixrank(E_A)=2$$
since $s\Big(\det E_A\Big)=x_1x_2\neq 0$. Therefore
$$\ELTrank(A)=2.$$

\end{example}

\begin{example}
Consider $A\in\mathbb{T}^{m\times n}$,
$$A=\begin{pmatrix}0 & 0 \\ 0 & 0 \end{pmatrix}.$$
Any matrix $E_A\in \big(\overline{\R^{\times}}\big)^{m\times n}$ such that $\tau(E_A)=A$ is of the form
$$E_A=\begin{pmatrix}\layer{0}{x_1} & \layer{0}{x_2} \\ \layer{0}{x_3} & \layer{0}{x_4} \end{pmatrix},$$
where $x_1,x_2,x_3,x_4\neq 0$.\\

Now for every choice of $E_A$, any submatrix of size $1\times 1$ is nonsingular; therefore
$$\ELTsubmatrixrank(E_A)\geq 1.$$
Choosing $x_1=x_2=x_3=x_4=1$ we obtain $E_A$ which is singular; and therefore
$$\ELTsubmatrixrank(E_A)= 1.$$

Together
$$\ELTrank(A)=1.$$

\end{example}

\begin{prop}\label{eltkapranovlem}
For any tropical matrix $A$
$$\ELTrank(A)= \Kapranovrank(A).$$
\end{prop}

\begin{proof}
Let $E_A\in \big(\overline{\R^{\times}}\big)^{m\times n}$ be a matrix such that $A=\tau(E_A)$. By \Lref{same_rank_lift}, there exists a matrix $A(t)\in \mathbb{K}^{m\times n}$ such that $\ELTrop\Big(A(t)\Big)=E_A$ and
$$\rank\Big(A(t)\Big) = \ELTsubmatrixrank(E_A).$$
Therefore by Definition \ref{kapranov_rank_def}
$$\ELTsubmatrixrank(E_A)\geq \Kapranovrank(A).$$
Since this is true for all such $E_A$, then
$$\ELTrank(A)\geq \Kapranovrank(A).$$

On the other hand, choose a matrix $A(t)\in \mathbb{K}^{m\times n}$ such that
$$\tau\Big(\ELTrop\big(A(t)\big)\Big)=A,$$
and
$$\rank\Big(A(t)\Big)= \Kapranovrank(A).$$
Since
$$\rank\Big(A(t)\Big)\geq \ELTsubmatrixrank\Big(\ELTrop\big(A(t)\big)\Big),$$
and
$$\ELTsubmatrixrank\Big(\ELTrop\big(A(t)\big)\Big)\geq \ELTrank(A),$$
then
$$\Kapranovrank(A)=\rank\Big(A(t)\Big)\geq \ELTrank(A).$$\\

Together we obtain
$$\Kapranovrank(A)=\ELTrank(A).$$
\end{proof}

\section{Inner Products and Orthogonality}

In this section, we introduce the definitions of ELT inner product and orthogonality.\break Although we prove that an orthogonal set of vectors is linearly independent, if we add an orthogonal vector to a linearly independent set, we may obtain a linearly dependent set.\\

\subsection{Inner product}

\begin{defn}
The \textbf{ELT conjugate} of $\layer{a}{\ell}\in\ELT{\mathbb{R}}{\mathbb{C}}$ is
$$\overline{\layer{a}{\ell}}=\layer{a}{\overline{\ell}}$$
where $\overline{\ell}$ is the usual conjugate in $\CC$. We define $\overline{-\infty}=-\infty$.
\end{defn}

\begin{defn}
Let $\R=\ELT{\mathbb{R}}{\mathbb{C}}$ be an ELT algebra. An \textbf{ELT inner product} is a function
$$\langle\cdot,\cdot\rangle:\left(\overline{\R}\right)^n\times \left(\overline{\R}\right)^n\rightarrow \overline{\R},$$
that satisfies the following three axioms for all vectors $v,u,w\in \left(\overline{\R}\right)^n$ and all scalars $a,b\in \overline{\R}$:\\
\begin{enumerate}
    \item $\langle av+bu,w\rangle=a\langle v,w\rangle+b\langle u,w\rangle$.
    \item $\left\langle v,u\right\rangle=\overline{\left\langle u,v\right\rangle}$.
    \item $s\Big(\langle v,v\rangle\Big)\geq 0_\mathbb{R}$ and if $v\in \left(\overline{\R^{\times}}\right)^n$, then $s\Big(\langle v,v\rangle\Big)=0_\mathbb{R}\iff v=(-\infty,...,-\infty)$.
\end{enumerate}
\end{defn}
Notice that we abuse the over-line notation for both an ELT algebra with the $-\infty$ element~($\overline{\R^{\times}}$) and the ELT conjugate.

\begin{example}

For any two vectors $v_1,v_2\in \left(\overline{\R}\right)^n$,
$$v_1=\left(\begin{matrix}\layer{\alpha_1}{z_1}\\\vdots\\\layer{\alpha_n}{z_n}\end{matrix}\right), v_2=\left(\begin{matrix}\layer{\beta_1}{w_1}\\\vdots\\\layer{\beta_n}{w_n}\end{matrix}\right)$$
we define the \textbf{standard} inner product
$$v_1\cdot v_2:=\layer{\alpha_1}{z_1}\layer{\beta_1}{\overline{w_1}}+...+\layer{\alpha_n}{z_n}\layer{\beta_n}{\overline{w_n}}.$$

The first two axioms are trivial to prove, and we will prove the third.\\

If $v=(-\infty,...,-\infty)$ then $v\cdot v=-\infty$, and thus $s\Big(v\cdot v\Big)=0_\mathbb{C}$.\\
Otherwise write
$$v=\left(\begin{matrix}\layer{\alpha_1}{z_1}\\\vdots\\\layer{\alpha_n}{z_n}\end{matrix}\right),$$
$$S=\{i|\alpha_i=\max_{1\leq j\leq n}\alpha_j\}.$$
Then
$$v\cdot v= \sum_{i\in S}\layer{2\alpha_i}{|z_i|^2},$$
and
$$s\Big(v\cdot v\Big)=\sum_{i\in S}|z_i|^2\geq 0_{\mathbb{R}}.$$
But $v\in \left(\overline{\R^{\times}}\right)^n$ and $v\neq (-\infty,...,-\infty)$, implying $z_i\neq 0_\mathbb{C}$ for all $i\in S$, and thus $s\Big(v\cdot v\Big)\neq 0_\mathbb{C}$.\\
\end{example}

As one can see, the tangible value of the standard inner product depends only on the tangible values of the input vectors. This, in fact, holds for every inner product:

\begin{lem}\label{lem:t-of_inn-prod-independent-of-layers}
If $u_1,u_2,v_1,v_2\in\left(\overline{\R}\right)^n$ satisfy $\t\left(u_1\right)=\t\left(u_2\right)$ and $\t\left(v_1\right)=\t\left(v_2\right)$, then
$$\t\left(\left\langle u_1,v_1\right\rangle\right)=\t\left(\left\langle u_2,v_2\right\rangle\right)$$
\end{lem}
\begin{proof}
We prove that if $\t\left(u_1\right)=\t\left(u_2\right)$, then
$$\t\left(\left\langle u_1,v\right\rangle\right)=\t\left(\left\langle u_2,v\right\rangle\right)$$
for any $v\in\left(\overline{\R}\right)^n$. The general assertion follows since $\t\left(\left\langle v,u\right\rangle\right)=\t\left(\left\langle u,v\right\rangle\right)$ for any $u,v$.

By \Lref{lem:t-equals-criterion}, there exist $w_1,w_2\in\left(\overline{\R}\right)^n$ such that $u_1=u_2+w_1$ and $u_2=u_1+w_1$. Thus,
\begin{eqnarray*}
  \left\langle u_1,v\right\rangle &=& \left\langle u_2+w_1,v\right\rangle=\left\langle u_2,v\right\rangle+\left\langle w_1,v\right\rangle \\
  \left\langle u_2,v\right\rangle &=& \left\langle u_1+w_2,v\right\rangle=\left\langle u_1,v\right\rangle+\left\langle w_2,v\right\rangle
\end{eqnarray*}
By \Lref{lem:t-equals-criterion}, the assertion is proved.
\end{proof}

As an immediate corollary, we get:

\begin{cor}
Let $v\in\left(\overline{\R}\right)^n$. If $\left\langle v,v\right\rangle=-\infty$, then
$$v=\left(-\infty,\dots,-\infty\right)$$
\end{cor}
\begin{proof}
By \Lref{lem:t-of_inn-prod-independent-of-layers}, we may assume that
$$s\left(v\right)=\left(1_{\CC},\dots,1_{\CC}\right)$$
(since changing the layers of the elements of $v$ will not affect the tangible value of $\left\langle v,v\right\rangle$, thus it will remain $-\infty$).Therefore we have $v\in\left(\overline{\R^{\times}}\right)^n$ with $\left\langle v,v\right\rangle=-\infty$. This immediately forces $v=\left(-\infty,\dots,-\infty\right)$, and the assertion is proved.
\end{proof}

\subsection{An ELT Cauchy-Schwarz Inequality}

We are now ready to prove the ELT version of Cauchy-Schwarz inequality.

\begin{thm}[ELT Cauchy-Schwarz Inequality]\label{thm:cauchy-schwarz}
Let $u,v\in\left(\overline{\R}\right)^n$. Then
$$\t\left(\left\langle u,v\right\rangle^2\right)\leq \t\left(\left\langle u,u\right\rangle\cdot \left\langle v,v\right\rangle\right)$$
\end{thm}
\begin{proof}
If either $u=\left(-\infty,\dots,-\infty\right)$ or $v=\left(-\infty,\dots,-\infty\right)$, the assertion is clear. Hence, we may assume that $u,v\neq\left(-\infty,\dots,-\infty\right)$. In addition, if $\left\langle u,v\right\rangle=-\infty$, the assertion is also trivial, so we assume that $\left\langle u,v\right\rangle\neq-\infty$.\\

We may change the layers of $u,v$ such that $u,v\in\left(\overline{\R^{\times}}\right)^n$ (since by \Lref{lem:t-of_inn-prod-independent-of-layers}, the tangible value of the inner products in the assertion will not be affected by a change in the layers). The same argument also allows us the assume that $s\left(\left\langle u,v\right\rangle\right)=1$. We note that since $s\left(\left\langle u,v\right\rangle\right)=1$, $\left\langle u,v\right\rangle=\left\langle v,u\right\rangle$.\\

Let $\lambda\in\overline{\R}$. Consider the ELT function
$$f\left(\lambda\right)=\left\langle\lambda u+v,\lambda u+v\right\rangle=\left\langle u,u\right\rangle\lambda\overline{\lambda}+\layer{0}{2}\left\langle u,v\right\rangle\lambda+\left\langle v,v\right\rangle.$$
By its definition, $s\left(f\left(\lambda\right)\right)\ge 0$ for any choice of $\lambda\in\overline{\R}$.\\

We claim that the middle monomial, $\layer{0}{2}\left\langle u,v\right\rangle\lambda$, cannot dominate the other two monomials at any point $\lambda$. Indeed, suppose that for $\lambda=\layer{\alpha}{\ell}$,
$$\t\left(\layer{0}{2}\left\langle u,v\right\rangle\lambda\right)>\max\left\{\t\left(\left\langle u,u\right\rangle\lambda^2\right), \t\left(\left\langle v,v\right\rangle\right)\right\}$$
Substituting $\lambda=\layer{\alpha}{-1}$ would yield $s\left(f\left(\layer{\alpha}{-1}\right)\right)=-2<0$, which is absurd.\\

We have proven that for any $\lambda\in\overline{\R}$,
$$\t\left(\layer{0}{2}\left\langle u,v\right\rangle\lambda\right)\leq\max\left\{\t\left(\left\langle u,u\right\rangle\lambda^2\right), \t\left(\left\langle v,v\right\rangle\right)\right\}$$
Equivalently,
$$\t\left(\left\langle u,v\right\rangle\right)+_{\RR}\t\left(\lambda\right)\leq\max\left\{\t\left(\left\langle u,u\right\rangle\right)+_{\RR}2\t\left(\lambda\right), \t\left(\left\langle v,v\right\rangle\right)\right\}.$$

Let us consider the points where the tangible values of $\left\langle u,u\right\rangle\lambda^2$ and $\left\langle v,v\right\rangle$ are equal. These are the points $\lambda\in\overline{\R}$ such that
$$\t\left(\lambda\right)=\frac{1}{2}\cdot_{\RR}\left(\t\left(\left\langle v,v\right\rangle\right)-_{\RR}\t\left(\left\langle u,u\right\rangle\right)\right).$$
Thus,
$$\t\left(\left\langle u,v\right\rangle\right)+_{\RR}\frac{1}{2}\cdot_{\RR}\left(\t\left(\left\langle v,v\right\rangle\right)-_{\RR}\t\left(\left\langle u,u\right\rangle\right)\right)\leq\t\left(\left\langle v,v\right\rangle\right)$$
which can be simplified to
$$\t\left(\left\langle u,v\right\rangle^2\right)=2\cdot_{\RR}\t\left(\left\langle u,v\right\rangle\right)\leq \t\left(\left\langle u,u\right\rangle\right)+_{\RR}\t\left(\left\langle v,v\right\rangle\right)=\t\left(\left\langle u,u\right\rangle\cdot \left\langle v,v\right\rangle\right)$$
as required.
\end{proof}

\begin{example}\label{exam:equality in CS with independent vectors}
There may be equality in the ELT Cauchy-Schwarz inequality even when $u,v$ are linearly independent. For example, take
$$u=\left(\begin{matrix}\layer{2}{1}\\\layer{0}{1}\end{matrix}\right), v=\left(\begin{matrix}\layer{2}{1}\\\layer{1}{1}\end{matrix}\right)$$
and equip $\overline{\R}^2$ with the standard inner product. The set $\left\{u,v\right\}$ is linearly independent, since
$$\det\left(\begin{matrix}\layer{2}{1}&\layer{2}{1}\\\layer{0}{1}&\layer{1}{1}\end{matrix}\right)=\layer{3}{1}$$
is not of layer zero. We note that
\begin{eqnarray*}
  u\cdot v &=& \layer{4}{1}+\layer{1}{1}=\layer{4}{1}\\
  u\cdot u &=& \layer{4}{1}+\layer{0}{1}=\layer{4}{1} \\
  v\cdot v &=& \layer{4}{1}+\layer{2}{1}=\layer{4}{1}
\end{eqnarray*}
Thus one can see that
$$\left(u\cdot v\right)^2=\layer{8}{1}=\left(u\cdot u\right)\left(v\cdot v\right)$$
even though the set $\left\{u,v\right\}$ is linearly independent.
\end{example}

\begin{example}
Furthermore, even if $\left\{u,v\right\}$ is a linearly dependent set, there may not be equality in the ELT Cauchy-Schwarz inequality. For instance, consider
$$u=\left(\begin{matrix}\layer{2}{0}\\\layer{0}{1}\end{matrix}\right), v=\left(\begin{matrix}\layer{0}{1}\\\layer{1}{0}\end{matrix}\right)$$
where we endow $\overline{\R}^2$ with the standard inner product. $\left\{u,v\right\}$ is a linearly dependent set, since
$$s\left(u_1+u_2\right)=s\left(\begin{matrix}\layer{2}{0}\\\layer{1}{0}\end{matrix}\right)= \left(\begin{matrix}0\\0\end{matrix}\right)$$
We calculate the inner products:
\begin{eqnarray*}
  u\cdot v &=& \layer{2}{0}+\layer{1}{0}=\layer{2}{0}\\
  u\cdot u &=& \layer{4}{0}+\layer{0}{1}=\layer{4}{0} \\
  v\cdot v &=& \layer{0}{1}+\layer{2}{0}=\layer{2}{0}
\end{eqnarray*}
Therefore
$$\t\left(\left(u\cdot v\right)^2\right)=\t\left(\layer{4}{0}\right)=4<6=\t\left(\layer{4}{0}\cdot\layer{2}{0}\right)=\t\left(\left(u\cdot u\right)\left(v\cdot v\right)\right)$$
\end{example}

However, we do have the following consolation:

\begin{thm}[Equality in the ELT Cauchy-Schwarz Inequality]
Let $u,v\in\left(\overline{\R}\right)^n$, and consider~$\left(\overline{\R}\right)^n$ with the standard inner product.
\begin{enumerate}
  \item We have equality of tangible values
  $$\t\left(\left(u\cdot v\right)^2\right)=\t\left(\left(u\cdot u\right)\left(v\cdot v\right)\right)$$
  if and only if there exists $1\leq k\leq n$ such that
  $$\t\left(u_k\right)=\max_{1\leq j\leq n}\t\left(u_j\right)$$
  and
  $$\t\left(v_k\right)=\max_{1\leq j\leq n}\t\left(v_j\right)$$
  In other words, there is equality of tangible values in the ELT Cauchy-Schwarz inequality if and only if the maximal tangible value in $u$ and $v$ is achieved in some common coordinate.
  \item Define $s_u,s_v\in\CC^n$ by
  $$\left(s_u\right)_i=\left\{\begin{matrix}s\left(u_i\right),&\t\left(u_i\right)=\max_{1\leq j\leq n}\t\left(u_j\right)\\0,&\textrm{Otherwise}\end{matrix}\right.$$
  and $s_v$ similarly for $v$. Then we have equality
  $$\left(u\cdot v\right)^2=\left(u\cdot u\right)\left(v\cdot v\right)$$
  if and only if $\t\left(\left(u\cdot v\right)^2\right)=\t\left(\left(u\cdot u\right)\left(v\cdot v\right)\right)$ and $s_u$ and $s_v$ are linearly dependent in $\CC^n$.
\end{enumerate}

\end{thm}
\begin{proof}~
\begin{enumerate}
  \item Using the definition of the standard inner product, we see that
  $$\t\left(\left(u\cdot v\right)\right)=\t\left(\sum_{i=1}^{n}u_i\overline{v_i}\right)=\max_{1\leq i\leq n}\t\left(u_i \overline{v_i}\right)=\max_{1\leq i\leq n}\left(\t\left(u_i\right)+_{\RR}\t\left(v_i\right)\right)$$
  which implies
  \begin{eqnarray*}
    \t\left(\left(u\cdot v\right)^2\right) &=& 2\cdot_{\RR}\max_{1\leq i\leq n}\left(\t\left(u_i\right)+_{\RR}\t\left(v_i\right)\right) \\
  \end{eqnarray*}
  Similarly,
  \begin{eqnarray*}
    \t\left(\left(u\cdot u\right)\right) &=& 2\cdot_{\RR}\max_{1\leq i\leq n}\t\left(u_i\right) \\
    \t\left(\left(v\cdot v\right)\right) &=& 2\cdot_{\RR}\max_{1\leq i\leq n}\t\left(v_i\right)
  \end{eqnarray*}

  Hence there is equality in Cauchy-Schwarz if and only if
  $$2\cdot_{\RR}\max_{1\leq i\leq n}\left(\t\left(u_i\right)+_{\RR}\t\left(v_i\right)\right)=2\cdot_{\RR}\max_{1\leq i\leq n}\t\left(u_i\right)+_{\RR}2\cdot_{\RR}\max_{1\leq i\leq n}\t\left(v_i\right)$$
  and this is equivalent to the above condition.

  \item For convenience, we define
  $$I=\left\{i\mid \t\left(u_i\right)=\max_{1\leq j\leq n}\t\left(u_j\right)\,\textrm{and}\,\t\left(v_i\right)=\max_{1\leq j\leq n}\t\left(v_j\right)\right\}.$$
  Therefore an equivalent formulation of the first part of this theorem is that
  $$\t\left(\left(u\cdot v\right)^2\right)=\t\left(\left(u\cdot u\right)\left(v\cdot v\right)\right)$$
  if and only if $I\neq\varnothing$.\\

  Now, if $I\neq\varnothing$, we have that
  $$s\left(u\cdot v\right)=s\left(\sum_{i=1}^{n}u_i\overline{v_i}\right)=\sum_{i\in I}s\left(u_i\right)\overline{s\left(v_i\right)}=s_u\cdot s_v$$
  where $s_u\cdot s_v$ is the standard inner product in $\CC^n$ of $s_u$ and $s_v$. Similarly,
  $$s\left(u\cdot u\right)=s_u\cdot s_u$$
  and
  $$s\left(v\cdot v\right)=s_v\cdot s_v$$
  Thus $$\left(u\cdot v\right)^2=\left(u\cdot u\right)\left(v\cdot v\right)$$ if and only if
  $$\t\left(\left(u\cdot v\right)^2\right)=\t\left(\left(u\cdot u\right)\left(v\cdot v\right)\right)$$
  and
  $$s\left(\left(u\cdot v\right)^2\right)=s\left(\left(u\cdot u\right)\left(v\cdot v\right)\right)$$
  if and only if $I\neq\varnothing$ and
  $$s\left(\left(u\cdot v\right)^2\right) = s\left(\left(u\cdot u\right)\left(v\cdot v\right)\right)$$
  if and only if $I\neq\varnothing$ and
  $$\left(s_u\cdot s_v\right)^2=\left(s_u\cdot s_u\right)\left(s_v\cdot s_v\right)$$
  if and only if $I\neq\varnothing$ and the set $\left\{s_u,s_v\right\}$ is linearly dependent (by the classical Cauchy-Schwarz inequality in $\CC^n$), as required.
\end{enumerate}
\end{proof}

We return to Cauchy-Schwarz inequality, and present the following corollary:

\begin{cor}\label{CS}
For any two vectors $u,v\in\left(\overline{\R}\right)^n$,
$$\t\left(\left\langle u,v\right\rangle\right)\leq\t\left(\left\langle u,u\right\rangle+\left\langle v,v\right\rangle\right)$$
In other words, either
$$\t\left(\left\langle u,v\right\rangle\right)\leq\t\left(\left\langle u,u\right\rangle\right)$$
or
$$\t\left(\left\langle u,v\right\rangle\right)\leq\t\left(\left\langle v,v\right\rangle\right)$$
\end{cor}
\begin{proof}
This immediately follows from the ELT Cauchy-Schwarz inequality, since
$$\t\left(\left\langle u,v\right\rangle\right)\leq\frac{1}{2}\cdot_{\RR}\left(\t\left(\left\langle u,u\right\rangle\right)+\t\left(\left\langle v,v\right\rangle\right)\right)\leq\max\left\{\t\left(\left\langle u,u\right\rangle\right),\t\left(\left\langle v,v\right\rangle\right)\right\}$$
\end{proof}

We extend this result to several vectors.\\

\begin{lem}\label{orth2}
If $v_1,...,v_k\in \left(\overline{\R}\right)^n$, then there exists some $p$ for which
$$\t\left(\left\langle v_p, v_p\right\rangle\right) \geq \t\left(\sum_{1\leq j\neq p \leq k}\left\langle v_j, v_p\right\rangle\right).$$
\end{lem}

\begin{proof}
Assume
$$\forall p :\t\left(\sum_{1\leq i,j \leq k}\left\langle v_i, v_j\right\rangle\right)> \t\left(\left\langle v_p, v_p\right\rangle\right),$$
and choose a specific $i\neq j$ such that $\forall p:\t\left(\left\langle v_i, v_j\right\rangle\right)>\t\left(\left\langle v_p, v_p\right\rangle\right)$. In particular, it follows that $\t\left(\left\langle v_i, v_j\right\rangle\right)>\t\left(\left\langle v_i, v_i\right\rangle\right)$. Thus by \Cref{CS} $\t\left(\left\langle v_i, v_j\right\rangle\right)\leq \t\left(\left\langle v_j, v_j\right\rangle\right)$, which contradicts our assumption.\\

Therefore there exists some $p$ for which
$$\t\left(\left\langle v_p, v_p\right\rangle\right)\geq \t\left(\sum_{1\leq i,j \leq k}\left\langle v_i, v_j\right\rangle\right),$$
and specifically
$$\t\left(\left\langle v_p, v_p\right\rangle\right) \geq \t\left(\sum_{1\leq j\neq p \leq k}\left\langle v_j, v_p\right\rangle\right).$$
\end{proof}

\subsection{Orthogonality}

\begin{defn}
Consider $u,v\in\left(\overline{\R}\right)^n$. We say $u,v$ are \textbf{ELT orthogonal} and write $u\perp v$ if
$$s\Big(\left\langle u,v\right\rangle\Big)=0_\mathbb{C}.$$
\end{defn}

\begin{thm}
If $v_1,...,v_k\in(\overline{\R^{\times}})^n$ are vectors such that
$$\forall i:v_i\neq (-\infty,...,-\infty)$$
and
$$\forall i\neq j:v_i\perp v_j,$$
then $v_1,...,v_k$ are linearly independent.\\
\end{thm}
\begin{proof}
Assume that $v_1,...,v_k$ are linearly dependent. Then there exists $\alpha_1,...,\alpha_k\in \R^{\times}$ such that
$$s\Big(\alpha_1v_1+...+\alpha_kv_k\Big)=(0_\mathbb{C},...,0_\mathbb{C}).$$
Therefore,
$$\alpha_1v_1+...+\alpha_kv_k=\layer{0}{0}\left(\alpha_1v_1+...+\alpha_kv_k\right)$$

If $u_i=\alpha_iv_i$, then by \Lref{orth2} there exists $p$ such that
$$\t\left(\left\langle u_p, u_p\right\rangle \right)\geq\t\left( \sum_{1\leq j\neq p \leq k}\left\langle u_j, u_p\right\rangle\right).$$

Multiplying by $u_p$, we obtain
\begin{eqnarray*}
  s\Big(\left\langle u_1, u_p\right\rangle + ... + \left\langle u_k, u_p\right\rangle \Big) &=& s\Big(\left\langle u_1+\cdots+u_k, u_p\right\rangle\Big)= \\
   &=& s\Big(\left\langle \layer{0}{0}\left(u_1+\cdots+u_k\right), u_p\right\rangle\Big)=\\
   &=& s\Big(\layer{0}{0}\left\langle u_1+\cdots+u_k, u_p\right\rangle\Big)=0_\CC
\end{eqnarray*}
Therefore $\forall i\neq p: s\Big(\left\langle u_i, u_p\right\rangle\Big)=0_\mathbb{C}$ and $\left\langle u_p, u_p\right\rangle$ dominates all other term. It follows that $s\Big(\left\langle u_p, u_p\right\rangle\Big)=0_\mathbb{C}$, which is absurd.
\end{proof}

We now aim to prove that every orthogonal set of size $k<n$ can be extended to an orthogonal set of size $n$. Before we do that, we need the concept of Gramian matrices.

\begin{defn}
Let $B=\left\{v_1,\dots,v_n\right\}$ be a subset of $\left(\overline{\R}\right)^n$. We define the \textbf{Gramian matrix} of~$\left\langle\cdot,\cdot\right\rangle$ with respect to $B$ as
$$G_B=\left(\begin{matrix}\left\langle v_1,v_1\right\rangle&\cdots&\left\langle v_1,v_n\right\rangle\\\vdots&\ddots&\vdots\\\left\langle v_n,v_1\right\rangle&\cdots&\left\langle v_n,v_n\right\rangle\end{matrix}\right)\in \left(\R\right)^{n\times n}$$
\end{defn}
\begin{rem}
As in the classical theory, we get two immediate facts about the Gramian matrix:
\begin{enumerate}
  \item If $B=\left\{e_1,\dots,e_n\right\}$ is the standard basis of $\left(\overline{\R}\right)^n$, then
$$\left\langle u,v\right\rangle=u^t G_B\overline{w}$$
where $\overline{w}$ is the vector in $\left(\overline{\R}\right)^n$ defined by $\left(\overline{w}\right)_i=\overline{w_i}$.
  \item For any set of vectors $B=\left\{v_1,\dots,v_n\right\}\subseteq\left(\overline{\R}\right)^n$, $G^t=\overline{G}$, where $\left(\overline{G}\right)_{i,j}=\overline{\left(G\right)_{i,j}}$.
\end{enumerate}
\end{rem}

\begin{lem}
If $G_B$ is nonsingular, then $B$ is linearly independent.
\end{lem}
\begin{proof}
Suppose that
$$s\left(\sum_{i=1}^n\alpha_i v_i\right)=\left(0_{\CC},\dots,0_{\CC}\right)$$
for some $\alpha_1,\dots,\alpha_n\in\overline{\R^{\times}}$, where not all the $\alpha_i$ are $-\infty$. Thus, for any $1\leq j\leq n$,
\begin{eqnarray*}
  s\left(\left\langle\sum_{i=1}^n\alpha_i v_i,v_j\right\rangle\right) &=& \left(0_{\CC},\dots,0_{\CC}\right) \\
  s\left(\sum_{i=1}^n\alpha_i\left\langle v_i,v_j\right\rangle\right) &=& \left(0_{\CC},\dots,0_{\CC}\right)
\end{eqnarray*}
This proves that the columns of $G_B$ are linearly dependent, in contradiction with \Tref{minusinftydet}.
\end{proof}

Unfortunately, the converse may not hold, as we see in the following example:
\begin{example}
We take $\overline{\R}^2$ with the standard inner product, and consider
$$u=\left(\begin{matrix}\layer{2}{1}\\\layer{0}{1}\end{matrix}\right), v=\left(\begin{matrix}\layer{2}{1}\\\layer{1}{1}\end{matrix}\right).$$
As we have seen in \Eref{exam:equality in CS with independent vectors}, $\left\{u,v\right\}$ is a linearly independent set, and
$$u\cdot u=u\cdot v=v\cdot v=\layer{4}{1}.$$
Thus
$$G_B=\left(\begin{matrix}\layer{4}{1}&\layer{4}{1}\\\layer{4}{1}&\layer{4}{1}\end{matrix}\right)$$
which is trivially singular.
\end{example}

\begin{lem}
Let $v_1,...,v_k\in (\overline{\R^{\times}})^n$ such that $k<n$, $v_i\perp v_j$ for all $i\neq j$, and $v_i\neq (-\infty,...,-\infty)$ for all $i$. Then there exist $v_{k+1},...,v_n$ such that the set $\{v_1,...,v_n\}$ is orthogonal.\\
\end{lem}

\begin{proof}
Let $S=\left\{e_1,\dots,e_n\right\}$ be the standard basis of $\left(\overline{\R}\right)^n$, and let $G$ be the Gramian matrix of $\left\langle\cdot,\cdot\right\rangle$ with respect to $S$. We construct $v_{k+1},\dots,v_n$ by induction.\\

Suppose that we have already constructed $\left\{v_1,\dots,v_\ell\right\}$ for $k\leq\ell<n$ such that $v_i\perp v_j$ for all $i\neq j$. We want to find $v_{\ell+1}\in\left(\overline{\R^{\times}}\right)^n$ such that $v_{\ell+1}\perp v_i$ for all $i\leq\ell$. If such $v_{\ell+1}$ exists, it should satisfy
$$s\left(\left\langle v_{\ell+1},v_i\right\rangle\right)=0 \Longleftrightarrow s\left(v_{\ell+1}^tG\overline{v_i}\right)=0\Longleftrightarrow s\left(\left(\overline{v_i}^t G^t\right)v_{\ell+1}\right)=0$$
for all $i\leq\ell$. Consider the matrix $A_\ell\in\left(\overline{\R}\right)^{\ell\times n}$, where
$$R_i\left(A_\ell\right)=\overline{v_i}^t G^t$$
for all $i\leq\ell$. Since $\ell<n$, the row rank of $A_\ell$ is at most $\ell$; by \Tref{rank}, its columns are linearly dependent. Hence there exists $v_{\ell+1}\in\left(\overline{\R^{\times}}\right)^n$ such that $s\left(A_\ell v_{\ell+1}\right)=\left(0_\CC,\dots,0_\CC\right)$; thus, by the above equivalences, $v_{\ell+1}\perp v_i$ for all $i\leq\ell$, as desired.
\end{proof}

\begin{example}
In this example, we present a linearly independent set $S=\{v_1,v_2\}$ which is not orthogonal, and a vector $v_3$ which is orthogonal to $S$, whereas the set $\{v_1,v_2,v_3\}$ is linearly dependent. We consider $\overline{\R}^3$ with the standard inner product, and
$$v_1=\left(\begin{matrix}\layer{2}{1}\\\layer{2}{-1}\\\layer{1}{-1}\end{matrix}\right), v_2=\left(\begin{matrix}\layer{2}{-1}\\\layer{2}{1}\\\layer{1}{-1}\end{matrix}\right), v_3=\left(\begin{matrix}\layer{1}{1}\\\layer{1}{1}\\\layer{1}{2}\end{matrix}\right).$$
These vectors are linearly dependent since
$$v_1+v_2+v_3=(\layer{2}{0},\layer{2}{0},\layer{1}{0}).$$
However, it is easy to see that $v_1,v_2$ are linearly independent, and that $v_3$ is orthogonal to both $v_1$ and $v_2$:
$$v_2\cdot v_3=v_1\cdot v_3=\layer{3}{0}.$$

\end{example}

\begin{defn}
Let $S=\left\{v_1,\dots,v_k\right\}\subseteq\left(\overline{\R}\right)^n$. We say that $S$ is \textbf{ELT orthonormal} if $v_i\perp v_j$ for any $i\neq j$ and $\left\langle v_i,v_i\right\rangle=\layer{0}{1}$.
\end{defn}

We can now prove an ELT version of Bessel's inequality:
\begin{thm}[ELT Bessel's Inequality]\label{thm:bessel-inequality}
Let $S=\left\{v_1,\dots,v_k\right\}\subseteq\left(\overline{\R}\right)^n$ be an ELT orthonormal set of vectors, and let $v\in\left(\overline{\R}\right)^n$. Let
$$u=\sum_{i=1}^{k}\left\langle v,v_i\right\rangle v_i$$
be the projection of $v$ on the subspace spanned by $S$. Then
$$\t\left(\left\langle u,u\right\rangle\right)\leq\t\left(\left\langle v,v\right\rangle\right)$$

Moreover, there is equality in Bessel's inequality if and only if there is some $1\leq i\leq n$ such that
$$\t\left(\left\langle v, v_i\right\rangle^2\right)=\t\left(\left\langle v,v\right\rangle\left\langle v_i,v_i\right\rangle\right)=\t\left(\left\langle v,v\right\rangle\right)$$
\end{thm}
\begin{proof}
By expanding the LHS,
$$\left\langle u,u\right\rangle=\sum_{i,j=1}^{k}\left\langle v,v_i\right\rangle\left\langle v,v_j\right\rangle\left\langle v_i,v_j\right\rangle=\sum_{i=1}^{k}\left\langle v,v_i\right\rangle^2+\sum_{\substack{i,j=1\\i\neq j}}^{k}\left\langle v,v_i\right\rangle\left\langle v,v_j\right\rangle\left\langle v_i,v_j\right\rangle.$$

We first show that the summands where $i\neq j$ do not contribute to the last sum. Indeed, if $i\neq j$, without loss of generality we assume $\t\left(\left\langle v,v_i\right\rangle\right)\leq\t\left(\left\langle v,v_j\right\rangle\right)$. By Cauchy-Schwarz inequality,
$$\t\left(\left\langle v_i,v_j\right\rangle\right)\leq\t\left(\left\langle v_i,v_i\right\rangle\left\langle v_j,v_j\right\rangle\right)=\t\left(\layer{0}{1}\cdot\layer{0}{1}\right)=0_{\RR}$$
Thus
$$\t\left(\left\langle v,v_i\right\rangle\left\langle v,v_j\right\rangle\left\langle v_i,v_j\right\rangle\right)=\t\left(\left\langle v,v_i\right\rangle\right)+_{\RR}\t\left(\left\langle v,v_j\right\rangle\right)+_{\RR}\t\left(\left\langle v_i,v_j\right\rangle\right)\leq\t\left(\left\langle v,v_j\right\rangle\right)+_{\RR}\t\left(\left\langle v,v_j\right\rangle\right)=\t\left(\left\langle v,v_j\right\rangle^2\right)$$
Since $s\left(\left\langle v,v_i\right\rangle\left\langle v,v_j\right\rangle\left\langle v_i,v_j\right\rangle\right)=0$, it will not contribute to the total sum. Hence
$$\left\langle u,u\right\rangle=\sum_{i=1}^{k}\left\langle v,v_i\right\rangle^2.$$

Finally, using Cauchy-Schwarz inequality,
$$\t\left(\left\langle u,u\right\rangle\right)=\t\left(\sum_{i=1}^{k}\left\langle v,v_i\right\rangle^2\right)=\max_{1\leq i\leq k}\t\left(\left\langle v,v_i\right\rangle^2\right)\leq\t\left(\left\langle v_i,v_i\right\rangle\left\langle v,v\right\rangle\right)=\t\left(\left\langle v,v\right\rangle\right)$$

The equality assertion is obvious from the above inequality.
\end{proof}

\bibliographystyle{plain}
\bibliography{references}

\end{document}